\documentclass[a4paper,11pt]{article}

\title{Moment based estimation for the multivariate COGARCH(1,1) process}
  
\author{
Thiago do R\^ego Sousa
\thanks{Center for Mathematical Sciences, Technical University of Munich,  85748 Garching, Boltzmannstr.~3, Germany, email:  thiago.sousa@tum.de}
\and Robert Stelzer 
\thanks{Institute of Mathematical Finance, Ulm University, 89075 Ulm, Helmholtzstr.~18, Germany, email:  robert.stelzer@uni-ulm.de}
}

  	   \usepackage[]{hyperref} 

       \usepackage{graphicx,todonotes}
       \usepackage{braket} 
       \usepackage{amsmath} 
       \usepackage{systeme} 
       \usepackage{mathtools} 
       \usepackage{bm} 
       \usepackage{amsfonts,amsthm,amstext,amssymb,amsmath}
       \usepackage[toc,page]{appendix}
       \usepackage{pdfpages}
       \usepackage{caption}
       \usepackage{color}
       \usepackage{comment}
       \usepackage[round]{natbib} 
       \usepackage{blkarray} 
	   \usepackage{multirow} 
	   \usepackage{dsfont}
	   \usepackage{enumitem}
	   \usepackage{kantlipsum} 

\numberwithin{equation}{section}

\newtheorem{theorem}{Theorem}[section]
\newtheorem{lemma}[theorem]{Lemma}
\newtheorem{fact}{Fact}[section]
\newtheorem{remark}[theorem]{Remark}
\newtheorem{example}[theorem]{Example}
\newtheorem{proposition}[theorem]{Proposition}
\newtheorem{definition}[theorem]{Definition}

\newtheorem{corollary}[theorem]{Corollary}

\newcommand{\bthe}{\begin{theorem}}
\newcommand{\ethe}{\end{theorem}}

\newcommand{\ble}{\begin{lemma}}
\newcommand{\ele}{\end{lemma}}

\newcommand{\bde}{\begin{definition}\rm}
\newcommand{\ede}{\Chalmos\end{definition}}

\newcommand{\bco}{\begin{corollary}}
\newcommand{\eco}{\end{corollary}}

\newcommand{\bpr}{\begin{proposition}}
\newcommand{\epr}{\end{proposition}}

\newcommand{\brem}{\begin{remark}\rm}
\newcommand{\erem}{\Chalmos\end{remark}}

\newcommand{\bproof}{\begin{proof}}
\newcommand{\eproof}{\end{proof}}

\newcommand{\bexam}{\begin{example}\rm}
\newcommand{\eexam}{\Chalmos\end{example}}

\newcommand{\beao}{\begin{eqnarray*}}
\newcommand{\eeao}{\end{eqnarray*}\noindent}

\newcommand{\beam}{\begin{eqnarray}}
\newcommand{\eeam}{\end{eqnarray}\noindent}

\newcommand{\barr}{\begin{array}}
\newcommand{\earr}{\end{array}}

\newcommand{\diff}{{\rm d}}
\def\disc{{\mathfrak d}}
\def\N{{\mathbb N}}
\def\C{{\mathbb C}}
\def\Z{{\mathbb Z}}

\def\E{{\mathbb E}}
\def\R{{\mathbb R}}
\def\S{{\mathbb S}}

\def\calb{{\mathcal{B}}}
\def\cala{{\mathcal{A}}}
\def\calr{{\mathcal{R}}}
\def\calc{{\mathcal{C}}}
\def\calq{{\mathcal{Q}}}
\def\calk{{\mathcal{K}}}

\def\calf{{\mathcal{F}}}

\def\calv{{\mathcal{V}}}

 \def\1{\mathds{1}}

\newcommand{\stp}{\stackrel{P}{\rightarrow}}
\newcommand{\std}{\stackrel{d}{\rightarrow}}
\newcommand{\stas}{\stackrel{\rm a.s.}{\rightarrow}}

\newcommand{\tto}{{t\to\infty}}

\newcommand{\nto}{{n\to\infty}}

\newcommand{\supthe}{\sup_{\theta \in \Theta}}      

\DeclareMathOperator*{\argminA}{arg\,min} 
\newcommand{\var}{{\rm var}}
\newcommand{\cov}{{\rm cov}}
\newcommand{\acov}{{\rm acov}}

\newcommand{\bsQ}{{\bm{Q}}}

\newcommand{\bsg}{{\bm{G}}}

\newcommand{\df}{{\rm d}}

\newcommand{\monedtwo}{M_{1,d^2}(\R)}

\newcommand{\mdone}{M_{d,1}(\R)}

\newcommand{\mdd}{M_{d,d}(\R)}
\newcommand{\md}{M_{d}(\R)}

\newcommand{\vshalf}{V_{s-}^{1/2}}
\newcommand{\vthalf}{V_{t-}^{1/2}}
\newcommand{\izede}{\int_0^{\Delta}}
\newcommand{\qvl}{[L,L^{\ast}]}
\newcommand{\vspr}{(V_s)_{s \in \R_{+}}}
\newcommand{\yspr}{(Y_s)_{s \in \R_{+}}}
\newcommand{\vtpr}{(V_t)_{t \in \R_{+}}}
\newcommand{\ytpr}{(Y_t)_{t \in \R_{+}}}
\newcommand{\ltpr}{(L_t)_{t \in \R_{+}}}

\newcommand{\partialthetai}{\frac{\partial}{\partial \theta_i}}
\newcommand{\partialtheta}{\frac{\partial}{\partial \theta}}

\newtheorem*{assumptiona}{Assumptions a}
\newtheorem*{assumptionb}{Assumptions b}
\newtheorem*{assumptionc}{Assumption c}
\newtheorem*{assumptiond}{Assumptions d}
\newtheorem*{assumptione}{Assumption e}
\newtheorem*{assumptionf}{Assumptions f}
\newtheorem*{assumptiong}{Assumptions g}

\DeclareMathOperator{\vect}{vec}
\DeclareMathOperator{\vech}{vech}

\newcommand{\Chalmos}{\quad\hfill\mbox{$\Box$}}  

\allowdisplaybreaks[3]

\begin{document}
\maketitle

\begin{abstract}
For the multivariate COGARCH process, we obtain explicit expressions for the second-order structure of the ``squared returns'' process observed on an equidistant grid. Based on this, we present a generalized method of moments estimator for its parameters. Under appropriate moment and strong mixing conditions, we show that the resulting estimator is consistent and asymptotically normal. Sufficient conditions for strong mixing, stationarity and identifiability of the model parameters are discussed in detail. We investigate the finite sample behavior of the estimator in a simulation study.
\end{abstract}

\noindent
{\em AMS 2010 Subject Classifications:} primary: 62M05, 62M10, secondary: 60G51, 91B84.\\
\noindent
{\em Keywords: multivariate continuous time GARCH, second-order moment structure, estimation, generalized method of moments, model identification, L\'{e}vy process}

\section{Introduction}

The modeling of financial data has received much attention over the last decades, where several models have been proposed for capturing its ``stylized facts''.
Prominent models are the class of ARCH (autoregressive conditionally heteroskedastic) and GARCH (generalized ARCH) processes introduced in \cite{Engle-DefinitionARCH-1982, Bollerslev-DefinitionGARCH-1986}. They are able to capture most of these stylized facts of financial data (see \cite{cont2001, guillaume1997bird}). A special feature of GARCH like processes is that they usually exhibit heavy tails even if the driving noise is light tailed, a feature most other stochastic volatility models do not have (\cite{fasen2006extremal}).

In many financial applications, it is most natural to model the price evolution in continuous time, especially when dealing with high-frequency data. The COGARCH process is a natural generalization of the discrete time GARCH process to continuous time. It exhibits many ``stylized features'' of financial time series and is well suited for modeling high-frequency data (see \cite{Bayraci14, Bibbona15, Haug07, Kluppelberg10, Maller08, Muller10}).



In many cases one needs to model the joint price of several financial assets which exhibit a non-trivial dependence structure and therefore, multivariate models are needed.
The MUCOGARCH process introduced in \cite{Stelzer10} is a multivariate extension of the COGARCH process. It combines the features of the continuous time GARCH processes with the ones of the multivariate BEKK GARCH process of \cite{engle1995BEKKGARCH}. It is a $d-$dimensional stochastic process and it is defined as
\begin{equation}\label{eq:def:Pt}
G_t = \int_0^t V_{s-}^{1/2} \diff L_s, \quad t\ge 0,
\end{equation}
where $L$ is an $\R^d$-valued L\'evy process with non-zero L\'{e}vy measure and c\`adl\`ag sample paths. The matrix-valued volatility process $(V_s)_{s \in \R^+}$  depends on a parameter $\theta \in \Theta \subset \R^q$, it is predictable and its randomness depends only on $L$. We assume that we have a sample of size $n$ of the log-price process \eqref{eq:def:Pt} with true parameter $\theta_0 \in \Theta$ observed on a fixed grid of size $\Delta>0$, and compute the log returns 
\begin{equation}\label{eq:def:Gn}
\bsg_i = \int_{(i-1)\Delta}^{i \Delta} V_{s-}^{1/2} \diff L_s, \quad i =1,\dots,n.
\end{equation}
Therefore, an important question is how to estimate the true parameter $\theta_0$ based on observations $(\bsg_i)_{i=1}^n$. In the univariate case, several methods have been proposed to estimate the parameters of the COGARCH process (\cite{Bayraci14, Bibbona15, Thiago1, Haug07, Maller08}). All these methods rely on the fact that the COGARCH process is, under certain regularity conditions, ergodic and strongly mixing. 

In the univariate case, \cite{fasen2010asymptotic} proved geometric ergodicity results for the COGARCH process (in fact, their results apply to a wider class of L\'evy driven models). Recently, \cite{Stelzer17} derived sufficient conditions for the existence of a unique stationary distribution, for the geometric ergodicity, and for the finiteness of moments of the stationary distribution in the MUCOGARCH process. These results imply ergodicity and strong mixing of the log-price process $(\bsg_i)_{i=1}^\infty$, thus paving the way for statistical inference. We will use their results to apply the generalized method of moments (GMM) for estimating the parameters of the MUCOGARCH process. To this end we compute the second-order structure of the squared returns in closed form, under appropriate assumptions.


Consistency and asymptotic normality of the GMM estimator is obtained under standard assumptions of strong mixing, existence of moments of the MUCOGARCH volatility process and model identifiability. Thus we discuss sufficient conditions, easily checkable for given parameter spaces ensuring strong mixing and existence of relevant moments.



The identifiability question is rather delicate, since the formulae for the second-order structure of the log-price returns involve operators which are not invertible and, therefore, the strategy used for showing identifiability as used in the one-dimensional COGARCH process cannot be generalised. In the end we can establish identifiability conditions that are not overly restrictive and easy to use.

Our paper is organized as follows. In Section 2, we fix the notation and briefly introduce L\'{e}vy processes. In Section 3 we define the MUCOGARCH process, and obtain in Section 4 its second-order structure. Section 5 introduces the GMM estimator and discusses sufficient conditions for stationarity, strong mixing and identifiability of the model. In Section~\ref{s:sim}, we study the finite sample behavior of the estimators in a simulation study. Finally, Section~\ref{se:proofs} presents the proofs for the results of Sections~\ref{se:MCOdef} and \ref{se:mFor}.

\section{Preliminaries}

\subsection{Notation}\label{s:notation}
Denote the set of non-negative real numbers by $\R^+$. For $z \in \C, \Re(z)$ and $\Im(z)$ denote the real and imaginary part, respectively. We denote by $M_{m,d}(\R)$, the set of real $m \times d$ matrices and write $M_d(\R)$ for $M_{d,d}(\R)$. The group of invertible $d \times d$ matrices is denoted by $GL_{d}(\R)$, the linear subspace of symmetric matrices by $\S_d$, the (closed) positive semidefinite cone by $\S^+_
d$ and the (open) positive definite cone by $\S^{++}_d$. We write $I_d$ for the $d \times d$ identity matrix. The tensor (Kronecker) product of two matrices $A,B$ is written as $A\otimes B$. The $\vect$ operator denotes the well-known vectorization operator that maps the set of $d\times d$ matrices to $\R^{d^2}$ by stacking the columns of the matrices below one another. Similarly, $\vech$ stacks the entries on and below the main diagonal of a square matrix. For more information regarding the tensor product, $\vect$ and $\vech$ operators we refer to \cite{Bernstein05, horn1991topics}. The spectrum of a square matrix is denoted by $\sigma(\cdot)$. Finally, $A^*$ denotes the transpose of a matrix $A \in M_{m,d}(\R)$ and $A_{(i,j)}$ denotes the entry in the $i$th line and $j$th column of $A$. Arbitrary norms of vectors or matrices are denoted by $\|\cdot\|$ in which case it is irrelevant which particular norm is used. The norm $\|\cdot\|_{2}$ denotes the operator norm on $M_{d^{2}}(\mathbb{R})$ associated with the usual Euclidean norm. The symbol $c$ stands for any positive constant, whose value may change from line to line, but is not of particular interest.

Additionally, we employ an intuitive notation with respect to (stochastic) integration with matrix-valued integrators, referring to any of the standard texts (for example, \cite{Protter90}) for a comprehensive treatment of the theory of stochastic integration. Let $(A_t )_{t\in \R^+}$ in $M_{m,d}(\R)$ and $(B_t )_{t\in \R^+}$ in $M_{r,u}(\R)$ be c\`adl\`ag and adapted processes and $(L_t )_{t\in \R^+}$ in $M_{d,r}(\R)$ be a semimartingale. We then denote by $\int_0^t A_{s-} \diff L_s B_{s-}$ the matrix $C_t \in M_{m,u}(\R)$ which has $ij$-th entry
$\sum_{k=1}^d \sum_{l=1}^{r} \int_{0}^{t} A_{i k, s-} B_{l j, s-} \mathrm{d} L_{k l, s}$. If $(X_t )_{t\in \R^+}$ is a semimartingale in $\R^m$ and $(Y_t )_{t\in \R^+}$ one in $\R^d$, then the quadratic variation $([X,Y]_t )_{t\in \R^+}$ is defined as the finite variation process in $M_{m,d}(\R)$ with $ij$-th entry $[X_i,Y_j ]_t$ for $t \in  \R^+$, $i = 1,\dots,m$ and $j = 1,\dots,d$. We also refer to Lemma~2.2 in \cite{Behme2012} for a collection of basic properties related to integration with matrix-valued integrators. Lastly, let $\bsQ:M_{d^2}(\R) \mapsto M_{d^2}(\R)$ be the linear map defined by
\begin{equation*}\label{eq:defQ}
(\bsQ X)_{(k-1)d+l,(p-1)d + q} = X_{(k-1)d+p,(l-1)d + q} \quad \text{ for all } k,l,p,q = {1,\dots,d},    
\end{equation*}
which has the property that $\bsQ (\vect(X)\vect(Z)^T) = X \otimes Z$ for all $X,Z \in \S_d$ (\cite[Theorem 4.3]{PigorschetStelzer2010MOU}). Let $K_d$ be the commutation matrix characterized by $K_d \vect(A) = \vect(A^*)$ for all $A \in M_d(\R)$ (see \cite{Magnus79Kom} for more details). Define $\calq \in M_{d^4}(\R)$ as the matrix associated with the linear map $\vect \circ \bsQ \circ \vect^{-1}$ on $\R^{d^4}$, and $\calk_d \in M_{d^4}(\R)$ as the matrix associated with the linear map $\vect(K_d \vect^{-1}(x))$ for $x \in \R^{d^4}$.

\subsection{L\'{e}vy processes}
A L\'{e}vy process $L = (L_t)_{t \in \R^+}$ in $\R^d$ is characterized by its characteristic function in L\'{e}vy-Khintchine form $\E e^{i \langle u, L_t \rangle}$ = $\exp\{ t \psi_{L}(u) \}$ for $t \in \R^{+}$ with
$$
\psi_{L}(u)=i \langle\gamma_{L}, u \rangle-\frac{1}{2} \langle u, \Gamma_{L} u \rangle+\int_{\mathbb{R}^{d}} \big(e^{i \langle u, x \rangle}-1-i\langle u, x\rangle I_{[0,1]}(\|x\|)\big) \nu_{L}(d x), \quad u \in \mathbb{R}^{d},
$$
where $\gamma_{L} \in \mathbb{R}^{d}, \Gamma_{L} \in \mathbb{S}_{d}^{+}$ and the L\'{e}vy measure $\nu_{L}$ is a non-zero measure on $\mathbb{R}^{d}$  satisfying $\nu_{L}(\{0\})=0$ and $\int_{\mathbb{R}^{d}}\left(\|x\|^{2} \wedge 1\right) \nu_{L}(d x)<\infty$. We assume w.l.o.g. $L$ to have c\`adl\`ag paths. The discontinuous part of the quadratic variation of $L$ is denoted by $([L,L]_t^{\disc})_{t \in \R^+}$ and it is also a L\'{e}vy process. It has finite variation, zero drift and L\'{e}vy measure $\nu_{[L, L]^{\disc}}(B)=\int_{\mathbb{R}^{d}} I_{B}\left(x x^{*}\right) \nu_{L}(d x)$ for all Borel sets $B \subseteq \mathbb{S}_{d}$. For more details on L\'{e}vy processes we refer to \cite{Applebaum09, Sato99}.

\section{The MUCOGARCH process}\label{se:MCOdef}

%
%

Throughout, we assume that all random variables and processes are defined on a given filtered probability space $(\Omega,\calf,P,(\calf_t )_{t\in T})$, with $T=\N$ in the discrete-time case and $T=\R^+$ in the continuous-time one. In the continuous-time setting, we assume the usual conditions (complete, right-continuous filtration) to be satisfied. We can now recall the definition of the MUCOGARCH process.


\begin{definition}[{MUCOGARCH(1,1) - \cite[Definition~3.1]{Stelzer10}}]\label{df:MUCOG} Let $L$ be an $\R^d$-valued L\'{e}vy process, $A,B \in M_d(\R)$ and $C \in \S_d^{++}$. The process $G = (G_t)_{t \in \R^+}$ solving
\beam
\diff G_t & = &  V_{t-}^{1/2} \diff L_t\label{eq:dGt}\\
V_t & = &  C + Y_t \label{eq:Vt}\\
\diff Y_t & = &  (B Y_{t-} + Y_{t-}B^\ast)\diff t + A V_{t-}^{1/2} \diff[L,L]_t^{\disc} V_{t-}^{1/2} A^\ast \label{eq:dYt}
\eeam
with initial values $G_0$ in $\R^d$ and $Y_0$ in $\mathbb{S}_d^{+}(\R)$ is called a MUCOGARCH(1,1) process. The process $Y = (Y_t)_{t \in \R^+}$ is called a MUCOGARCH(1,1) volatility process. Hereafter we will always write MUCOGARCH for short.
\end{definition}

The interpretation of the model parameters $B$ and $C$ is the following. If $\sigma(B) \in \{ z \in \C: \Re(z)<0 \}$, the process $V$, as long as no jump occurs, ``mean reverts'' to the level $C$ at matrix exponential rate given by $B$. Since all jumps are positive semidefinite, $C$ is not a mean level but, instead, a lower bound for $V$.

By \cite[Theorems~3.2 and 4.4]{Stelzer10}, the MUCOGARCH process is well-defined, the solution $(Y_t )_{t\in \R^+}$ is locally bounded and of finite variation. Additionally, the process $(G_t,Y_t )_{t\in \R^+}$ and its volatility process $(Y_t )_{t\in \R^+}$ are time homogeneous strong Markov processes on $\R^d \times \S^+_d$ and $\S^+_d$, respectively.

Since the price process $(G_t)_{t \in \R^+}$ in \eqref{eq:dGt} is defined in terms of the L\'evy process $L$ and $\ytpr$, the existence of its moments is closely related to the existence of moments of $L$ and the stationary distribution of $\ytpr$.

\begin{lemma}\label{le:momG1}
Suppose that $\E \|Y_0\|^p < \infty$ and $\E \|L_1\|^{2p} < \infty$ for some $p \geq 1$. Then:
\begin{itemize}
\item[(a)] $\E \|Y_t\|^p < \infty$ for all $t \in \R^+$ and $t \mapsto \E \|Y_t\|^p$ is locally bounded.
\item[(b)] $\E \|G_t\|^{2p} < \infty$ for all $t \in \R^+$ and $t \mapsto \E \|G_t\|^{2p}$ is locally bounded.
\end{itemize}
\end{lemma}

\section{Second-order structure of ``squared returns''}\label{se:mFor}


In this section, we derive the second-order structure of the MUCOGARCH ``squared returns'' process $(\bsg_i\bsg_i^\ast)_{i \in \N}$ defined in terms of \eqref{eq:def:Gn}, which will be used in Section~\ref{se:GMM} to estimate the parameters $A,B$ and $C$ of the MUCOGARCH process. The proofs are postponed to Section~\ref{se:proofs}. We group the needed assumptions as follows.





\begin{assumptiona}[L\'{e}vy process]
\mbox{}
\begin{enumerate}[label=$(a.\arabic*)$]

\item \label{as:Elzero} $\E L_1 = 0$.

\item \label{as:varL1_id} $\var (L_1) = (\sigma_W + \sigma_L ) I_d$, with $\sigma_W \geq 0$ and $\sigma_L > 0$.

\item \label{as:lp_qv}
\begin{equation*}
 \int_{\R^d}x_ix_jx_k \, \nu_L(dx)=0, \quad \text{for all }  i,j,k \in \{1,\dots,d\}.
\end{equation*}

\item \label{as:El14fin} $\E \|L_1\|^4 < \infty$.

\item \label{as:Equadquad} There exists a constant $\rho_L > 0$ such that
$$
\E [ \vect(\qvl^\disc), \vect(\qvl^\disc)^{\ast} ]_1^\disc =  \rho_L ( I_{d^2} + K_{d} + \vect(I_d) \vect(I_d)^{\ast} ).
$$

\item \label{as:El18fin} $\E \|L_1\|^8 < \infty$.

\end{enumerate}
\end{assumptiona}

\begin{assumptionb}[Parameters]
\mbox{}
\begin{enumerate}[label=$(b.\arabic*)$]
\item  $A \in GL_d(\R)$.

\item \label{as:BCcurlInv} The matrices $\calb$ and $\calc$ defined below satisfy $\sigma(\calb), \sigma(\calc) \in \{ z \in \C: \Re(z)<0 \}$.
\beam
   \calb & :=  & B\otimes I + I \otimes B + \sigma_L(A \otimes A) \label{eq:def:Calb} \\
    \calc &  :=  &\calb \otimes I_{d^2} +  I_{d^2}\otimes\calb + \cala \calr \nonumber,
\eeam
where $\cala = (A \otimes A) \otimes (A \otimes A)$, $\calr = \rho_L (\calq + \calk_d \calq + I_{d^4})$, and $\calk_d$ and $\calq$ as in Section~\ref{s:notation}.


\end{enumerate}
\end{assumptionb}

\begin{assumptionc}[MUCOGARCH volatility]
\mbox{}
\begin{enumerate}[label=$(c.\arabic*)$]

\item\label{as:YtsecSt} $\ytpr$ is a second-order stationary MUCOGARCH volatility process.

\item\label{as:Y0E4} $\ytpr$ is a stationary MUCOGARCH volatility process and its stationary distribution satisfies $\E \|Y_0\|^4 < \infty$.
\end{enumerate}
\end{assumptionc}

Sufficient conditions for Assumption \textbf{c} are given in \cite[Theorem~4.5]{Stelzer10}. Note that \ref{as:Y0E4} implies \ref{as:YtsecSt}. We recall now the expressions for the second-order structure of the process $Y$ and of the log-price returns process $(\bsg_i)_{i \in \N}$. First, for a second-order stationary $\R^d$-valued process, its autocovariance function $\acov_X:\R \mapsto M_d(\R)$ is denoted by $\operatorname{acov}_{X}(h)=\operatorname{cov}\left(X_{h}, X_{0}\right)=\E\left(X_{h} X_{0}^{*}\right)-\E\left(X_{0}\right) E\left(X_{0}\right)^{*}$ for $h \geq 0$ and by $\operatorname{acov}_{X}(h)=\left(\operatorname{acov}_{X}(-h)\right)^{*}$ for $h < 0$. For matrix-valued processes $(Z_t)_{t \in \R}$, we set 
$\acov_Z = \acov_{\vect(Z)}$. 

\begin{proposition}[{\cite[Theorems~4.8, 4.11, Corollary~4.19 and Proposition~5.2]{Stelzer10}}]\label{pr:SecStrV0}
If Assumptions \ref{as:Elzero}-\ref{as:Equadquad}, \ref{as:BCcurlInv} and \ref{as:YtsecSt} hold, then
\beam
\E(\vect(Y_0)) &=& -\sigma_L \calb^{-1}(A \otimes A )\vect(C)\label{eq:EV}\\
\var(\vect(Y_0))  &=&  \var(\vect(V_0)) = -\calc^{-1}\big[\big(\sigma_L^2 \calc(\calb^{-1}\otimes \calb^{-1})\cala + \cala \calr\big)(\vect(C)\otimes \vect(C)) \nonumber \\
& & \quad\quad + \big(\sigma_L(A \otimes A) \otimes I_{d^2} + \cala  \calr\big)\vect(C) \otimes \E (\vect(Y_0)) \nonumber \\
& & \quad\quad + \big(\sigma_L I_{d^2} \otimes (A \otimes A)  + \cala  \calr\big) \E (\vect(Y_0)) \otimes \vect(C) \big]\nonumber\\
\acov_Y(h) &=& \acov_V(h) = e^{\calb h} \var(\vect(Y_0))\nonumber \\
\E(\bsg_1) &=& 0\nonumber\\
\var(\bsg_1) &=& (\sigma_L+\sigma_W) \Delta \E (C + Y_0)\label{eq:EG1}\\
\acov_\bsg(h) &=& 0 \quad \text{for all }h \in \Z\backslash \{0\}\nonumber.
\eeam
\end{proposition}

Based on Lemma~\ref{le:momG1} and Proposition~\ref{pr:SecStrV0}, we obtain now the second-order properties of the MUCOGARCH process.

\begin{lemma}\label{le:acovgg}
If Assumptions \textbf{a},\textbf{b} and \textbf{c} hold, then

\begin{equation}\label{eq:acvgg}
\begin{split}
\acov_{\bm{G}\bm{G}^{\ast}}(h) & = e^{\calb \Delta h} \calb^{-1}(I_{d^2} - e^{-\calb \Delta})(\sigma_L + \sigma_W) \var (\vect (V_0))  \\
& \quad\quad \times  (e^{\calb^{\ast}\Delta} - I_{d^2})[ (\sigma_W + \sigma_L)(\calb^{\ast})^{-1} - 2((A \otimes A)^{\ast})^{-1} ], \quad h \in \N,
\end{split}
\end{equation}
\beam
&  & \E \vect(\bm{G}_1 \bm{G}_1^{\ast}) \vect(\bm{G}_1 \bm{G}_1^{\ast})^{\ast} \label{eq:vg1g1}\\ 
&  & = \Delta \rho_L  \big( (\bsQ +K_{d} \bsQ  +  I_{d^2})(\E \vect(V_0) \vect(V_0)^{\ast})\big) \times \nonumber\\
& &  \,\,\,\,\,\, (I_{d^2} + K_d) \bsQ (D^\ast) (I_{d^2} + K_d) + D + D^\ast \nonumber,
\eeam
with,
\begin{equation}\label{eq:E}
D:= (\sigma_L + \sigma_W) \Big( \frac{1}{2}(\sigma_L + \sigma_W)\Delta^2\E \vect(V_0)   \E \vect( V_0 )^{\ast}  +  \var (\vect (V_0)) \tilde{\calb} \Big)
\end{equation}
\begin{equation}\label{eq:Btilde}
\tilde{\calb}:=    \big[ (\calb^{\ast})^{-1}( e^{\calb^{\ast}\Delta} - I_{d^2}) - I_{d^2}\Delta \big] \big[ (\sigma_W + \sigma_L)(\calb^{\ast})^{-1} - 2((A \otimes A)^{\ast})^{-1}  \big]
\end{equation}
\end{lemma}

\begin{remark}
If the L\'evy process $L$ has paths of finite variation, then Lemma~\ref{le:acovgg} holds without the moment assumptions \ref{as:El18fin} and \ref{as:Y0E4}. This is because expectations involving stochastic integrals with finite variation L\'evy integrators can be computed by using the compensation formula (see Remark \ref{re:new}). In the following, we will define the moment based estimator for MUCOGARCH processes driven by general L\'evy processes (without path restrictions). Only in Section~\ref{se:areEA} we will give a consistency result that distinguishes between L\'evy process with paths of finite and infinite variation.
\end{remark}
Next, we define an estimator for the parameters $A,B$ and $C$, which basically consists of comparing the sample moments to the model moments.



\section{Moment based estimation of the MUCOGARCH process}\label{se:GMM}

In this section, we consider the matrices $A_\theta,B_\theta \in M_{d}(\R)$ and $C_\theta \in \mathbb{S}_{d}^{++}$ from Definition~\ref{df:MUCOG} as depending on a parameter $\theta \in \Theta \subset \R^q$ for $q \in \N$.

The data used for estimation is an equidistant sample of $d$-dimensional log-prices $(\bsg_i)_{i = 1}^n$ as defined in \eqref{eq:def:Gn} with true parameter $\theta_0 \in \Theta$. We assume that the true $\sigma_L, \sigma_W$ and $\rho_L$ as used in Assumptions~\ref{as:varL1_id} and \ref{as:Equadquad} are known. These assumptions are not very restrictive and are comparable to assuming iid standard normal noise in the discrete time multivariate GARCH process, which is very common \cite[eq. (10.6)]{Francq11}.

\subsection{Generalized Method of Moments (GMM) estimator}\label{se:gmm}

In order to estimate the parameter $\theta_0 \in \Theta$, we compare the sample moments (based on a sample of log-prices) to the model moments (based on the expressions \eqref{eq:EG1}, \eqref{eq:acvgg} and \eqref{eq:vg1g1}, provided they are well defined). More specifically, based on the observations $(\bsg_i)_{i = 1}^n$ and a fixed $r < n$, the sample moments are defined as
\begin{equation}\label{eq:defkn}
\hat{k}_{n,r} =  \frac{1}{n} \sum_{i=1}^{n-r} D_i =  \frac{1}{n} \sum_{i=1}^{n-r}
\begin{pmatrix}
\vect(\bm{G}_i \bm{G}_i^{\ast}) \\
\vect(\vect(\bm{G}_i \bm{G}_i^{\ast})\vect(\bm{G}_{i} \bm{G}_{i}^{\ast})^*) \\
\vdots \\
\vect(\vect(\bm{G}_i \bm{G}_i^{\ast})\vect(\bm{G}_{i+r} \bm{G}_{i+r}^{\ast})^*)
\end{pmatrix}.
\end{equation}
The used number of lags of the true autocovariance function $r$ needs to be chosen in such a way that the model parameters are identifiable and also to ensure a good fit of the autocovariance structure to the data. For each $\theta \in \Theta$, let 
\begin{equation}\label{eq:defEk1}
k_{\theta,r} = 
\begin{pmatrix}
\E_\theta \vect(\bm{G}_1 \bm{G}_1^{\ast}) \\
\E_\theta \vect(\vect(\bm{G}_1 \bm{G}_1^{\ast})\vect(\bm{G}_1 \bm{G}_1^{\ast})^*)  \\
\vdots \\
\E_\theta \vect(\vect(\bm{G}_1 \bm{G}_1^{\ast})\vect(\bm{G}_{1+r} \bm{G}_{1+r}^{\ast})^*) 
\end{pmatrix},
\end{equation}
where the expectations are explicitly given by \eqref{eq:EG1}, \eqref{eq:acvgg} and \eqref{eq:vg1g1} by replacing $A,B$ and $C$ by $A_\theta,B_\theta$ and $C_\theta$, respectively. Then, the GMM estimator of $\theta_0$ is given by
\begin{equation}\label{eq:defGMM}
\hat{\theta}_n = \argminA_{\theta \in \Theta} \Big\{ (\hat{k}_{n,r} - k_{\theta,r})^T\Omega (\hat{k}_{n,r} - k_{\theta,r}) \Big\},
\end{equation}
where $\Omega$ is a positive definite weight matrix. The matrix $\Omega$ may be depend on the data but should converge in probability to a positive definite matrix of constants.

\subsection{Asymptotic properties: general case}\label{se:apgmmG}

Additionally to Assumptions~\textbf{a}, \textbf{b} and \textbf{c} we need assumptions for proving consistency and asymptotic normality of $\hat{\theta}_n$. These are mainly related to identifiability of the model parameters, stationarity, strong mixing and existence of certain moments of $(\bsg_i)_{i \in \N}$.

\begin{assumptiond}[Parameter space and log-price process]
\mbox{}
\begin{enumerate}[label=$(d.\arabic*)$]

\item \label{as:comp} The parameter space $\Theta$ is a compact subset of $\R^q$.

\item \label{as:theInt} The true parameter $\theta_0$ lies in the interior of $\Theta$.

\item \label{as:ident}[Identifiability]. Let $r > 1$ be fixed. For any $\theta \not= \tilde{\theta} \in \Theta$ we have $k_{\theta,r} \not= k_{\tilde{\theta},r}$.

\item \label{as:diffABC} The map $\theta \mapsto (A_\theta,B_\theta,C_\theta)$ is continuously differentiable.

\item \label{as:mixing} The sequence $(\bsg_i)_{i \in \N}$ is strictly stationary and exponentially $\alpha$-mixing.

\end{enumerate}

\end{assumptiond}

\begin{assumptione}[Moments]
\mbox{}
\begin{enumerate}[label=$(e.\arabic*)$]
\item \label{as:momANC} There exists a positive constant $\delta > 0$ such that $\E \|\bsg_1\|^{8 + \delta} < \infty$.

\end{enumerate}

\end{assumptione}

Assumption~\textbf{e} can be written in terms of moments of $L$ and $Y_0$ (see Lemma~\ref{le:momG1}). We are now ready to state the strong consistency of the empirical moments in \eqref{eq:defkn}. 
\begin{lemma}\label{le:ConsM}
If Assumptions \textbf{a}, \textbf{b}, \textbf{c} and \ref{as:mixing} hold, then $\hat{k}_{n,r} \stas k_{\theta_0}$ as $\nto$.
\end{lemma}
\begin{proof}
It follows from \ref{as:mixing} that the log-price process $(\bsg_i)_{i \in \N}$ is ergodic and since both \\ $\E \|\vect(\bm{G}_1 \bm{G}_1^{\ast})\|$ and $\E \|\vect(\bm{G}_1 \bm{G}_1^{\ast}) \vect(\bm{G}_{1+h} \bm{G}_{1+h}^{\ast})^*\|$ are finite (Lemma~\ref{le:momG1} with $p=2$ under \ref{as:El14fin} and \ref{as:YtsecSt}), we can apply Birkhoff's ergodic theorem (\cite[Theorem 4.4]{Krengel85}) to conclude the result.
\end{proof}

Next, we state the weak consistency property of the GMM estimator.
\begin{theorem}\label{th:conGMM}
If Assumptions \textbf{a},\textbf{b}, \textbf{c}, \ref{as:comp}, \ref{as:ident}-\ref{as:mixing} hold, then the GMM estimator defined in \eqref{eq:defGMM} is weakly consistent. 
\end{theorem}
\begin{proof}
We check Assumptions~1.1-1.3 in \cite{matyas99GMM} that ensure weak consistency of the GMM estimator in \eqref{eq:defGMM}. Assumption~1.1 is satisfied due to our identifiability condition \ref{as:ident}. It follows from \eqref{eq:defGMM} combined with Lemma~\ref{le:ConsM} that 
\begin{equation*}
\sup_{\theta \in \Theta} \| \hat{k}_{n,r} - k_{\theta,r} - (k_{\theta_0,r} - k_{\theta,r}) \| =  \| \hat{k}_{n,r}  - k_{\theta_0,r}  \|\stas 0, \quad \nto,
\end{equation*}
which is Assumption~1.2 of \cite{matyas99GMM}. Since the weight matrix $\Omega$ in \eqref{eq:defGMM} is non-random, their Assumption~1.3 is automatically satisfied, completing the proof.
\end{proof}

In order to prove asymptotic normality of the GMM estimator, we need some auxiliary results.
\begin{lemma}\label{le:contDifk}
If Assumptions \textbf{a}, \textbf{b},  \textbf{c}, \ref{as:comp} and \ref{as:diffABC} hold, then the map $\Theta \mapsto k_{\theta,r}$ in \eqref{eq:defEk1} is continuously differentiable.
\end{lemma}
\begin{proof}
The the map $\Theta \mapsto k_{\theta,r}$ depends on the moments given in \eqref{eq:EG1}, \eqref{eq:acvgg} and \eqref{eq:vg1g1}. These moments are given in terms of products and Kronecker products involving the quantities $A_\theta$, $A_\theta^{-1}$, $\calb_\theta$, $\calb_\theta^{-1}$, $e^{-\alpha \calb_\theta}$, $\alpha > 0$, $C_\theta$, $\calc_\theta$ and $\calc_\theta^{-1}$. From \ref{as:diffABC} we obtain the continuous differentiability of $\calb_\theta$,$\calb_\theta^{-1}$, $\calc_\theta$,$\calc_\theta^{-1}$ and $A_\theta^{-1}$ on $\Theta$. Let $i \in \{1,\dots,q\}$ be fixed. According to (2.1) in \cite{Wilcox67}, the matrix exponential is differentiable and
\begin{equation}\label{eq:difEbe}
\partialthetai e^{-\alpha \calb_\theta} = - \int_0^\alpha e^{-(\alpha-u)\calb_\theta} \bigg( \partialthetai \calb_\theta \bigg) e^{-u \calb_\theta} \diff u.
\end{equation}
Using the definition of $\calb_\theta$ in \eqref{eq:def:Calb} combined with \ref{as:comp} and \ref{as:diffABC} gives
\begin{equation*}\label{eq:bsupnb}
\supthe \|\calb_\theta\| \leq 2 \bigg( \supthe \|B_\theta\| \bigg) \|I_d\| + \sigma_L \bigg( \supthe \|A_\theta\|^2 \bigg) < \infty.
\end{equation*}
Additionally, an application of the chain rule to $\partialthetai \calb_\theta$ combined with \ref{as:comp} and \ref{as:diffABC} gives $\supthe \| \partialthetai \calb_\theta \| < \infty$ and, therefore,
\begin{equation}\label{eq:bsupndb}
\supthe \bigg\| e^{-(\alpha-u)\calb_\theta} \bigg( \partialthetai \calb_\theta \bigg) e^{-u \calb_\theta} \bigg\| \leq \supthe e^{(|\alpha-u| + |u|)\|\calb_\theta\|} \bigg( \supthe \bigg\| \partialthetai \calb_\theta \bigg\| \bigg), \quad u \in [0,\alpha].
\end{equation}
Thus, the continuous differentiability of the map in \eqref{eq:difEbe} follows by dominated convergence with dominating function as in \eqref{eq:bsupndb}. Another application of the chain rule shows that the map $\theta \mapsto k_{\theta,r}$ is continuously differentiable on $\Theta$. 
\end{proof}

\begin{lemma}\label{le:asnmom}
Assume that Assumptions~\textbf{a},\textbf{b}, \textbf{c}, \ref{as:mixing} and \ref{as:momANC} hold and let 
\begin{equation}\label{eq:defSigma}
\Sigma_{\theta_0}  =  \E (F_1 F_1^\ast) + \sum_{i=1}^{\infty} \E  \big\{(F_1 F_{1+i}^\ast) + \E   (F_1 F_{1+i}^\ast)^\ast \big\} 
\end{equation}

with $F_i = D_i - k_{\theta_0,r}$ and $D_i$ as defined in \eqref{eq:defkn}. Then for $r \in \N_0$
\begin{equation*}\label{eq:asn}
\sqrt{n}(\hat{k}_{n,r} - k_{\theta_0,r}) \std \mathcal{N}(0,\Sigma_{\theta_0}), \quad n \rightarrow \infty.
\end{equation*}
\end{lemma}
\begin{proof}
For the asymptotic normality of \eqref{eq:defkn} we use the Cram\'er-Wold device and show that
\begin{equation*}
\sqrt{n}\Big( \frac{1}{n} \sum_{i=1}^{n-r} \lambda^\ast F_i \Big) \std \mathcal{N}(0,\lambda^\ast \Sigma_{\theta_0} \lambda),\quad\nto,
\end{equation*}
for all vectors $\lambda \in \R^{d^2 + (r+1)d^4}$. Denote by $\alpha_{\bsg}$ the mixing coefficients of $(\bsg_i)_{i \in \N}$. Since each $F_i$ is a measurable function of $\bsg_i,\dots,\bsg_{i+r}$ it follows from \ref{as:mixing} and Remark~1.8 of \cite{Bradley07} that $(\lambda^* F_i)_{i \in \N}$ is $\alpha$-mixing with mixing coefficients satisfying $\alpha_{F}(n) \leq \alpha_{\bsg}(n-(r+1))$ for all $n \geq r + 2$. Therefore, $\sum_{n=0}^{\infty} (\alpha_{F}(n))^{\frac{\epsilon}{2 + \epsilon}} < \infty$ for all $\epsilon > 0$. 
From \ref{as:momANC} we obtain $\E \| \lambda^\ast F_1 \|^{2+\epsilon/4} < \infty$ for some $\epsilon > 0$. 
Thus, the CLT for $\alpha$-mixing sequences applies, see e.g. \cite[Theorem~18.5.3]{Ibragimov71}, so that 
\begin{equation*}
\sqrt{n}\Big( \frac{1}{n} \sum_{i=1}^{n-r} \lambda^\ast F_i \Big) \std \mathcal{N}(0,\zeta),\quad\nto,
\end{equation*}
where
\begin{equation*}
\zeta = \E \lambda^\ast F_1 F_1^\ast \lambda  + 2\sum_{i=1}^{\infty} \E \lambda^\ast F_1 F_{1+i}^\ast \lambda.
\end{equation*}
Since $\lambda^\ast F_1 F_{1+i}^\ast \lambda = \lambda^\ast (F_1 F_{1+i}^\ast)^\ast \lambda$  we get \eqref{eq:defSigma} after rearranging the above equation.
\end{proof}

\begin{theorem}\label{th:NormGMM}
Assume that Assumptions~\textbf{a},\textbf{b}, \textbf{c}, \textbf{d} and \ref{as:momANC} hold and that the matrix $\Sigma$ in \eqref{eq:defSigma} is positive definite. Then the GMM estimator defined in \eqref{eq:defGMM} is asymptotically normal with covariance matrix
\begin{equation}\label{eq:acmGMM}
(\mathcal{J}_{\theta_0})^{-1}\mathcal{I}_{\theta_0}(\mathcal{J}_{\theta_0})^{-1},
\end{equation}
where
$\mathcal{J}_{\theta_0} = (\nabla_{\theta} k_{\theta_0,r})^\top \Omega (\nabla_{\theta} k_{\theta_0,r})$ and 
$\mathcal{I}_{\theta_0} = (\nabla_{\theta} k_{\theta_0,r})^\top\Omega\Sigma_{\theta_0}\Omega(\nabla_{\theta} k_{\theta_0,r})$.

\end{theorem}
\begin{proof}
We check Assumptions~1.7-1.9 of Theorem~1.2 in \cite{matyas99GMM}. Since by Lemma~\ref{le:contDifk} the map $\theta \mapsto k_{\theta,r}$ is continuously differentiable, their Assumption~1.7 is valid. Now, for any sequence $\tilde{\theta}_n$ such that $\tilde{\theta}_n \stp \theta_0$ as $\nto$, it follows from the continuous mapping theorem by the continuity of the map $\Theta \mapsto \partialtheta k_{\theta,r}$ in Lemma~\ref{le:contDifk} that $\partialtheta (\hat{k}_{n,r} - k_{\theta_n}) \stp (k_{\theta_0} - \partialtheta k_{\theta_0})$ as $\nto$. Therefore, Assumption~1.8 in \cite{matyas99GMM} is also satisfied. Since Lemma~\ref{le:asnmom} implies Assumption~1.9, we conclude the result.
\end{proof}


\begin{remark}
In order to apply the results of Section~\ref{se:apgmmG} we need to check Assumption~\textbf{c}, model identifiability \ref{as:ident}, strong mixing of the log-price returns sequence \ref{as:mixing} and existence of certain moments of its stationary distribution (Assumption~\textbf{e}) . 
In Sections~\ref{se:mixing} and \ref{se:identif} we give sufficient conditions for identifiability of the model parameters, strict stationarity and strong mixing. Then we use these results to derive in Section~\ref{se:areEA} more palpable conditions under which Theorems~\ref{th:conGMM} and \ref{th:NormGMM} can be applied.
\end{remark}

\subsection{Sufficient conditions for strict stationarity and strong mixing}\label{se:mixing}

Sufficient conditions for the existence of a unique stationary distribution of $(Y_t)_{t \in \R^+}$, geometric ergodicity and for the finiteness of moments of order $p$ of the stationary distribution have recently been given in \cite{Stelzer17}. We state these conditions in the next theorem, which are conditions (i), (iv) and (v) of Theorem~4.3 in \cite{Stelzer17}.

\begin{theorem}[{Geometric Ergodicity - \cite[Theorem~4.3]{Stelzer17}}]\label{th:331Vest} Let $Y$ be a MUCOGARCH volatility process which is $\mu$-irreducible with the support of $\mu$ having non-empty interior and aperiodic. Assume that one of the following conditions is satisfied:
\begin{itemize}

\item[(i)] setting \(p=1\) there exists \(\Xi \in \mathbb{S}_{d}^{++}\) such that
\begin{equation}\label{eq:cond41}
\Xi B+B^{\top} \Xi+\sigma_L A^{\top} \Xi A  \in-\mathbb{S}_{d}^{++},
\end{equation}
\item[(ii)] there exist $p \in[1, \infty)$ and $\Xi \in \mathbb{S}_{d}^{++}$ such that
\begin{equation}\label{eq:thcii}
\int_{\mathbb{R}^{d}}\left(2^{p-1}\left(1+K_{\Xi, A}\|y\|_{2}^{2}\right)^{p}-1\right) \nu_{L}(d y)+ p K_{\Xi, B} <0,
\end{equation}
where 
$$
K_{\Xi, B} =\max _{X \in \mathbb{S}_{d}^{+}, \operatorname{tr}(X)=1} \frac{\operatorname{tr}\left(\left(\Xi B+B^{\top} \Xi\right) X\right)}{\operatorname{tr}(\Xi X)} \text { and } K_{\Xi, A} =\max _{X \in \mathbb{S}_{d}^{+}, \operatorname{tr}(X)=1} \frac{\operatorname{tr}\left(A^{\top} \Xi A X\right)}{\operatorname{tr}(\Xi X)},
$$

\item[(iii)] there exist $p \in[1, \infty)$ and $\Xi \in \mathbb{S}_{d}^{++}$ such that
\begin{equation}\label{eq:thciii}
\max \left\{2^{p-2}, 1\right\} K_{\Xi, A} \int_{\mathbb{R}^{d}}\|y\|_{2}^{2}\left(1+\|y\|_{2}^{2} K_{\Xi, A}\right)^{p-1} \nu_{L}(d y)+K_{\Xi, B}<0
\end{equation}
where $K_{\Xi, B}, K_{\Xi, A}$ are as in (ii).
\end{itemize}
Then a unique stationary distribution for the MUCOGARCH volatility process $Y$ exists, $Y$ is positive Harris recurrent, geometrically ergodic and its stationary distribution has a finite $p$-th moment.
\end{theorem}

A consequence of Theorem~\ref{th:331Vest} is that the process $Y$ is exponentially $\beta$-mixing. This implies $\alpha$-mixing of the log-price process as we state next. For more details on mixing conditions we refer to \cite{Bradley07}. 

\begin{corollary}\label{co:expBmix}
If $Y$ is strictly stationary and exponentially $\beta$-mixing, then the log-price process $(\bsg_i)_{i \in \N}$ is stationary, exponentially $\alpha$-mixing, and as a consequence also ergodic.
\end{corollary}
\begin{proof}
Since $Y$ is an exponentially $\beta$-mixing, homogeneous strong Markov process \cite[Theorem~4.4]{Stelzer10}, and driven only by the discrete part of the quadratic variation of $L$, the proof follows by the same arguments as for Theorem~3.4 in \cite{Haug07}.
\end{proof}

Next, we state a result which gives sufficient conditions for the irreducibility of the MUCOGARCH volatility process $Y$ process, which is one of the sufficient conditions for the geometric ergodicity result in Theorem~\ref{th:331Vest}.

\begin{theorem}[{Irreducibility and Aperiodicity - \cite[Theorem~5.1 and Corollary ~5.2]{Stelzer10}}]\label{th:Th51JS}
Let $Y$ be a MUCOGARCH volatility process driven by a L\'evy process whose discrete part is a compound Poisson process $L$ with $A \in G L_{d}(\mathbb{R})$ and  $\Re(\sigma(B))<0$. If the jump distribution of L has a non-trivial absolutely continuous component equivalent to the Lebesgue measure on $\R^d$ restricted to an open neighborhood of zero, then $Y$ is irreducible w.r.t. the Lebesgue measure restricted to an open neighborhood of zero in $
\mathbb{S}_{d}^{+}$ and aperiodic.

\end{theorem}





\subsection{Sufficient conditions for identifiability}\label{se:identif}

In this Section we investigate the identifiability of the model parameters from the model moments, i.e., we investigate the injectivity of the map $\theta \mapsto k_{\theta,r}$ on an appropriate compact set $\Theta$. Recall that we can divide the this map into the composition of $\theta \mapsto (A_\theta,B_\theta,C_\theta) \mapsto k_{\theta,r}$. Injectivity of $\theta \mapsto (A_\theta,B_\theta,C_\theta)$ holds if e.g. it simply maps the entries of $\theta$ to the entries of the matrices $(A_\theta,B_\theta,C_\theta)$. Thus, we only need to investigate the injectivity of the map $(A_\theta,B_\theta,C_\theta) \mapsto k_{\theta,r}$. As we will see, there will appear some restrictions on the matrices $A_\theta,B_\theta$, which are related to the fact that we need to take the logarithm of a matrix exponential, and we need to ensure this is well defined. We will omit $\theta$ from the notation, except when explicitly needed. We start with the identifiability of the matrix $C$.
\begin{lemma}\label{le:idC}
Assume that Assumptions \ref{as:Elzero}-\ref{as:Equadquad}, \ref{as:BCcurlInv} and \ref{as:YtsecSt} hold and that $\sigma(B) \subset \{ z \in \C: \Re(z)<0 \}$. If the matrices $A$ and $B$ are known, then $\E(\bsg_1 \bsg_1^\ast)$ uniquely determines $C$. 
\end{lemma}
\begin{proof}
Since $\sigma(B\otimes I + I \otimes B) = \sigma(B) + \sigma(B) \subset \{ z \in \C: \Re(z)<0 \}$, the matrix $B\otimes I + I \otimes B$ is invertible. The rest of the proof follows by noting that from \eqref{eq:EV} and \eqref{eq:EG1} it follows that
\begin{equation*}\label{eq:EG1forC}
\vect(C)  = (\sigma_L+\sigma_W)^{-1} \Delta^{-1} (B\otimes I + I \otimes B)^{-1}\calb\vect(  \E(\bsg_1 \bsg_1^\ast)).
\end{equation*}
\end{proof}
For the identification of the matrices $A$ and $B$ we need to use the second-order structure of the squared returns process in Lemma~\ref{le:acovgg}. We first state three auxiliary results, which provide conditions such that we can identify the components of the autocovariance function in \eqref{eq:acvgg}.

\begin{lemma}
Assume that $B \in M_d(\R)$ is diagonalizable with $S \in GL_d(\C)$ such that $S^{-1}BS$ is diagonal. If
\begin{equation}\label{eq:baainv}
\bigg| \frac{\sigma_L - \sigma_W}{2\sigma_L}\bigg| \|A \otimes A \|_{S}< -2 \max\{ \Re(\sigma(B))\},
\end{equation}
with
\begin{equation}\label{eq:defXYn}
\|X\|_{S} = \|(S^{-1} \otimes S^{-1}) X (S \otimes S))\|_2, \quad 
X \in M_{d^2}(\R), S \in GL_d(\C),
\end{equation}
then the matrix 
\begin{equation}\label{eq:bm2a}
(\sigma_W + \sigma_L)(\calb^{\ast})^{-1} - 2((A \otimes A)^{\ast})^{-1}
\end{equation}
is invertible.
\end{lemma}
\begin{proof}
From \cite[fact~2.16.14]{Bernstein05}, $X^{-1} + Y^{-1}$ is non-singular if and only if $X + Y$ is non-singular and $X,Y$ are non-singular. Setting $X = \frac{\calb}{(\sigma_L + \sigma_W)}, Y = -\frac{1}{2}(A \otimes A)$ and using the definition of $\calb$ in \eqref{eq:def:Calb} we get
$$
X + Y = \frac{1}{(\sigma_L+\sigma_W)}\bigg((B \otimes I + I \otimes B) + \frac{(\sigma_L-\sigma_W)}{2}(A \otimes A)\bigg).
$$
Since $B$ is diagonalizable, we can use \cite[Proposition~7.1.6]{Bernstein05} to obtain
$$
B \otimes I + I \otimes B = (S \otimes S )( S^{-1}BS \otimes I ) (S^{-1} \otimes S^{-1}),
$$
which guarantees that $B \otimes I + I \otimes B$ is also diagonalizable. Now we rewrite the first equation on p. 106 in \cite{Stelzer10} with the matrix $\calb$ replaced by $(B \otimes I + I \otimes B) + \frac{(\sigma_L-\sigma_W)}{2}(A \otimes A)$ and apply the Bauer-Fike Theorem \cite[Theorem~6.3.2]{horn1991topics} to see that \eqref{eq:baainv} implies that all eigenvalues of $(X+Y)(\sigma_L+\sigma_W)$ are in $\{ z \in \C: \Re(z)<0 \}$ and, therefore, $X + Y$ is invertible.
\end{proof}

\begin{lemma}\label{le:auxLem}
If $A \in \md$ is such that $A_{(1,1)},\dots,A_{(1,{j-1})} = 0$ and $A_{(1,j)} > 0$ for some $j \in \{1,\dots,d\}$, then the map $X\mapsto AXA^T$ for $X \in \S_d$ identifies $A$.
\begin{proof}
Assume first that $A_{(1,1)} > 0$. For each $i \in \{1,\dots,d\}$, let $e_i$ be the $i$th column unit vector in $\R^d$ and define the matrix $E^{(i,j)} = e_i e_j ^T$. The first line of the matrix $A E^{(1,1)} A^T$ is
\begin{equation}\label{eq:1line}
(A_{(1,1)}^2 , A_{(1,1)}A_{(2,1)}, \dots, A_{(1,1)}A_{(d,1)}).
\end{equation}
Since $A_{(1,1)} > 0$, \eqref{eq:1line} allows us to identify first $A_{(1,1)}$ and then $A_{(2,1)},\dots,A_{(d,1)}$. Now, for each $k \in \{2,\dots,d\}$, note that $E^{(1,k)} + E^{(k,1)}$ is symmetric. Simple calculations reveal that the first line of the matrix $A (E^{(1,k)} + E^{(k,1)}) A^T$ is
\begin{equation}\label{eq:1line2}
(2 A_{(1,1)}A_{(1,k)}, A_{(1,1)}A_{(2,k)} + A_{(1,k)}A_{(2,1)}, \dots, A_{(1,1)}A_{(d,k)} + A_{(1,k)}A_{(d,1)}).
\end{equation}
Since $A_{(1,1)} > 0$, we identify $A_{(1,k)}$ from the first entry of \eqref{eq:1line2}. Now, since also
$A_{(2,1)},\dots,A_{(d,1)}$ are already known, we can identify $A_{(2,k)},\dots,A_{(d,k)}$. Thus, all entries of A can be identified. The cases $A_{(1,j)} > 0$ for some $j > 1$ follow similarly.
\end{proof}
\end{lemma}

\begin{lemma}\label{le:IdeExp}
Assume that the Assumptions \textbf{a},\textbf{b} and \textbf{c} and the conditions of Lemma~\ref{le:auxLem} hold, that the matrix in \eqref{eq:bm2a} is invertible, that $\sigma(\calb) \subset \{ z \in \C: -\pi < \Im(z)\Delta < \pi, \Re(z)<0 \}$ and that $\var(\vech(V_0))$ is invertible. Define $M = (e^{\calb \Delta})^{-1}\acov_{\theta,\bm{G}\bm{G}^{\ast}}(1)$. Then, $\acov_{\theta,\bm{G}\bm{G}^{\ast}}(1)$ and
$\acov_{\theta,\bm{G}\bm{G}^{\ast}}(2)$ uniquely identify $\calb$ and $M$.
\end{lemma}
\begin{proof}
Since $M$ is given in terms of $\calb$ and $\acov_{\theta,\bm{G}\bm{G}^{\ast}}(1)$, we only need to identify $\calb$. Observe that we are using the $\vect$ operator only for convenience, as it interacts nicely with tensor products of matrices and thus gives nicely looking formulae. However, the volatility and ``squared returns'' processes take values in $\mathbb{S}_d$ which is a $d(d+1)/2$-dimensional vector space, whereas the $\vect$ operator assumes values in a $d^2$-dimensional vector space. Instead of using the $\vech$ operator and cumbersome notation, we take an abstract point of view. The variance of a random element of $\mathbb{S}_d$ is a symmetric positive semi-definite linear operator from  $\mathbb{S}_d$ to itself. Likewise, the autocovariance of $\bm{G}_1\bm{G}_1^{\ast}$ and $\bm{G}_{1+h}\bm{G}_{1+h}^{\ast}$ is a linear operator from $\mathbb{S}_d$ to itself. The condition that $\var(\vech(V_0))$ is invertible is equivalent to the invertibility of the linear operator, which is the variance of $V_0$. Similarly all other $d^2\times d^2$ matrices in
$$e^{\calb \Delta h} \calb^{-1}(I_{d^2} - e^{-\calb \Delta})(\sigma_L + \sigma_W) \var (\vect (V_0))  
  (e^{\calb^{\ast}\Delta} - I_{d^2})[ (\sigma_W + \sigma_L)(\calb^{\ast})^{-1} - 2((A \otimes A)^{\ast})^{-1} ]$$  are representing linear operators from  $\mathbb{S}_d$ to itself. Under the assumptions made, the above product involves only invertible linear operators. Hence $\acov_{\theta,\bm{G}\bm{G}^{\ast}}(h)$ is invertible (over $\mathbb{S}_d$) for every $h>0$. Thus,
$$
e^{\calb \Delta } =\acov_{\theta,\bm{G}\bm{G}^{\ast}}(2) [\acov_{\theta,\bm{G}\bm{G}^{\ast}}(1)]^{-1}.
$$
By the assumptions on the eigenvalues of $\calb$ there is a unique logarithm for $e^{\calb \Delta }$ (see \cite[Section 6.4]{horn1991topics} or \cite[Lemma 3.11]{Schlemm12}), so $\calb \Delta$ and thus $\calb$ is identified. Finally, note that the matrices in the $\vect$ representations are uniquely identified by the employed linear operators on $\mathbb{S}_d$ due to \cite[Proposition 3.1]{PigorschetStelzer2007b} and Lemma~\ref{le:auxLem}.
\end{proof}

\begin{lemma}
[Identifiability of $A$, $B$ and $C$]\label{le:IdABFrCB} For all $\theta \in \Theta$, assume the conditions of Lemma~\ref{le:IdeExp}, $\sigma(B_\theta) \subset \{ z \in \C: \Re(z)<0 \}$ and that the entries of the matrices $A_\theta$ and $B_\theta$ satisfy: for some $k \not= l \in \{1,\dots,d\}$,  $A_{(k,l),\theta} > 0$, $A_{(k,l),\theta} \not= A_{(l,k),\theta}$ and $B_{(k,l),\theta} = B_{(l,k),\theta}$. Then $k_{\theta,2}$ uniquely identifies $A_\theta$, $B_\theta$ and $C_\theta$.
\end{lemma}
\begin{proof}
Recall that we omit $\theta$ in the notation. Assume w.l.o.g that $\sigma_L = 1$. Because of Lemma~\ref{le:idC}, we only need to show the identification of $A$ and $B$.\\ 
Assume first that $d = 2$. Then the matrix $\calb$ from \eqref{eq:def:Calb} is
\begin{footnotesize}
\begin{equation}\label{eq:calb2}
\begin{pmatrix*}[r]
2B_{(1,1)} + A_{(1,1)}^2  & B_{(1,2)} + A_{(1,1)}A_{(1,2)} & B_{(1,2)} + A_{(1,1)}A_{(1,2)} & A_{(1,2)}^2 \\
B_{(2,1)} + A_{(1,1)}A_{(2,1)}  & B_{(1,1)} + B_{(2,2)} + A_{(1,1)}A_{(2,2)} & A_{(1,2)}A_{(2,1)}  & B_{(1,2)} + A_{(1,2)}A_{(2,2)} \\
B_{(2,1)} + A_{(1,1)}A_{(2,1)}  &  A_{(1,2)}A_{(2,1)} & B_{(1,1)} + B_{(2,2)} + A_{(1,1)}A_{(2,2)}  & B_{(1,2)} + A_{(1,2)}A_{(2,2)} \\
A_{(2,1)}^2   &  B_{(2,1)} + A_{(2,1)}A_{(2,2)} & B_{(2,1)} + A_{(2,1)}A_{(2,2)} & 2B_{(2,2)} + A_{(2,2)}^2
\end{pmatrix*}.
\end{equation}
\end{footnotesize}
Using the entry at position $(1,4)$ and the fact that $A_{(1,2)} > 0$ allow us to identify $A_{(1,2)}$. Then, we use the entry at position $(2,3)$ to identify $A_{(2,1)}$. Now, we use the entries at positions $(1,2)$ and $(2,1)$ together with the fact that $A_{(1,2)} \not= A_{(2,1)}$ and $B_{(1,2)} = B_{(2,1)}$ to write $A_{(1,1)} = (\calb_{(1,2)} - \calb_{(2,1)}) / (A_{(1,2)} -A_{(2,1)})$. Similarly we use the entries at positions $(3,4), (4,3)$ to get $A_{(2,2)} = (\calb_{(3,4)} - \calb_{(4,3)}) / (A_{(1,2)} -A_{(2,1)})$. Now, since all the entries of $A$ are known, we can use the entries at positions $(1,1), (1,2)$ and $(2,2)$ to identify the entries of $B$.\\
Now assume that $d > 2$. We assume w.l.o.g. that $A_{(1,2)} > 0$, $A_{(1,2)} \not= A_{(2,1)}$ and $B_{(1,2)} = A_{(2,1)}$. Write the matrix $\calb$ from \eqref{eq:def:Calb} in the following block form:
\begin{equation}\label{eq:Bblocks}
  \calb = B\otimes I + I \otimes B + A \otimes A =  \begin{pmatrix}
  \calb^{(1,1)} & \cdots &  \calb^{(1,d)} \\
  \vdots & \ddots & \vdots \\
  \calb^{(d,1)} & \cdots &  \calb^{(d,d)}
  \end{pmatrix},
\end{equation}
where $\calb^{(i,j)} \in M_{d}(\R)$ for all $i,j=1,\dots,d$. First, we have that
\begin{equation}\label{eq:bcurlblock1d}
   \calb^{(1,2)} =  \begin{pmatrix*}[r]
B_{(1,2)} + A_{(1,2)}A_{(1,1)}  & A_{(1,2)}A_{(1,2)} & A_{(1,2)}A_{(1,3)} & \cdots &   A_{(1,2)}A_{(1,d)} \\
  A_{(1,2)}A_{(2,1)}  & B_{(1,2)} +  A_{(1,2)}A_{(2,2)} & A_{(1,2)}A_{(2,3)} & \cdots &   A_{(1,2)}A_{(2,d)} \\  
  \vdots & \vdots & \vdots & \ddots & \vdots \\
A_{(1,2)}A_{(d,1)}  &   A_{(1,2)}A_{(d,2)} & A_{(1,2)}A_{(d,3)} & \cdots &   B_{(1,2)} + A_{(1,2)}A_{(d,d)}   
  \end{pmatrix*},
\end{equation}
Since $A_{(1,2)} > 0$ we can identify it from \eqref{eq:bcurlblock1d}, because  $ \calb^{(1,2)}_{(1,2)} = A_{(1,2)}^2$. Then we use the off-diagonal entries of the matrix $\calb^{(1,2)}$ in \eqref{eq:bcurlblock1d} together with $A_{(1,2)}$ to identify all the off-diagonal entries of the matrix $A$. Next we identify the diagonal entries of $A$. It follows from \eqref{eq:Bblocks} that  
\begin{equation}\label{eq:sysle}
\systeme*{\calb^{(k,k)}_{(1,2)} = B_{(1,2)} + A_{(k,k)}A_{(1,2)},\calb^{(k,k)}_{(2,1)} = B_{(2,1)} + A_{(k,k)}A_{(2,1)}}, \quad k = 1,\dots,d.
\end{equation}
Since $A_{(1,2)}-A_{(2,1)} \not= 0$ and $B_{(1,2)} = B_{(2,1)}$, the system of equations \eqref{eq:sysle} gives
$$
A_{(k,k)} = (\calb^{(k,k)}_{(1,2)} - \calb^{(k,k)}_{(2,1)}) / (A_{(1,2)} -A_{(2,1)}), \quad k = 1,\dots,d.
$$
Finally, since the matrix $A$ is now completely known, we can use \eqref{eq:Bblocks} to identify all entries of $B$. 
\end{proof}

In Lemma~\ref{le:IdABFrCB} we identify the matrices $A$ and $B$ only from $\calb$ and, therefore, some mild restrictions on the off-diagonal entries of $B$ appear. In order to avoid those restrictions, we could to take the structure of $\E \vect(\vect(\bm{G}_1 \bm{G}_1^{\ast})\vect(\bm{G}_1 \bm{G}_1^{\ast})^*)$ in \eqref{eq:vg1g1} into account when proving identifiability and we expect that one can improve the identification results since more moment conditions are used. However, already in the $2$-dimensional case the results on identification conditions are quite involved, and this has mainly to do with the fact that the linear operator $(\bsQ +K_{d} \bsQ  +  I_{d^2})$ at the right hand side of \eqref{eq:vg1g1} is not one-to-one in the space of matrices of the form $\E \vect(V_0) \vect(V_0)^{\ast}$. In the end, in order to use the moment conditions $\E \vect(\vect(\bm{G}_1 \bm{G}_1^{\ast})\vect(\bm{G}_1 \bm{G}_1^{\ast})^*)$, we need to assume that the matrices $B \otimes I + I \otimes B$ and $A \otimes A$ commute (see \cite[Lemma~3.5.18]{Thiago}). Since commutativity is a quite strong condition, it seems highly preferable to work with the class of MUCOGARCH processes, which are identifiable by Lemma~\ref{le:IdABFrCB}. The exponential decay of the autocovariance function of the model is still quite flexible, because of the interplay between the matrices $A$ and $B$ (see \eqref{eq:calb2}, for instance).

\subsection{Asymptotic properties: general case revisited}\label{se:areEA}

Here, we combine the results of Sections~\ref{se:apgmmG}-\ref{se:identif} to give easily verifiable conditions under which the GMM estimator $\hat{\theta}_n$ will be consistent and asymptotically normal. We assume that the parameter $\theta$ contains the entries of the matrices $(A_\theta,B_\theta,C_\theta)$ so that the map $\theta \mapsto (A_\theta,B_\theta,C_\theta)$ is automatically injective and continuously differentiable on $\Theta$.


First, we define
\beam 
\|x\|_{S} &= & \|(S^{-1} \otimes S^{-1}) x\|_2, \quad x \in \R^{d^2}, S \in GL_d(\C) \label{eq:defxYn}\\ 
K_{2,S} &= & \max_{ X \in \mathbb{S}^+_d, \|X\|_2 = 1} \bigg( \frac{\|X\|_2}{\|\vect(X)\|_{S}} \bigg), \quad S \in GL_d(\C) \label{eq:defK2Y}.
\eeam
Consider now the following group of assumptions:

\begin{assumptionf}[Parameter space]
For all $\theta \in \Theta$ it holds:
\mbox{}
\begin{enumerate}[label=$(f.\arabic*)$]

\item\label{as:par6}  The matrices $B_\theta$ satisfy $\sigma(B_\theta) \subset \{ z \in \C:  \Re(z)<0 \}$.

\item\label{as:par1} The matrix $\calb_\theta$ satisfy $\sigma(\calb_\theta) \subset \{ z \in \C: -\pi < \Im(z)\Delta < \pi, \Re(z)<0 \}$.

\item\label{as:par2}  The matrix $B_\theta \in M_d(\R)$ is diagonalizable with $S_\theta \in GL_d(\C)$ such that $S_\theta^{-1}B_\theta S_\theta$ is diagonal.

\item \label{as:par3}  The entries of the matrices $A_\theta$ and $B_\theta$ satisfy: for some $k \not= l \in \{1,\dots,d\}$,  $A_{(k,l),\theta} > 0$, $A_{(k,l),\theta} \not= A_{(l,k),\theta}$ and $B_{(k,l),\theta} = B_{(l,k),\theta}$.
\item\label{as:par4} The matrix $\var_\theta(\vech(V_0))$ is invertible.
\item\label{as:par5} 
$| \frac{\sigma_L - \sigma_W}{2\sigma_L}| \|A_\theta \otimes A_\theta \|_{S_\theta}< -2 \max\{ \Re(\sigma(B_\theta))\}$ with $S_\theta$ as in \ref{as:par2} and $\|A_\theta \otimes A_\theta \|_{S_\theta}$ as in \eqref{eq:defXYn}.

\item\label{as:par7} There exists \(\Xi_\theta \in \mathbb{S}_{d}^{++}\) such that, condition \eqref{eq:cond41} holds with $A,B$ replaced by $A_\theta,B_\theta$.


\item\label{as:par10} $m(4,\theta) < 0$ where 
\begin{equation}\label{eq:con44}
m(p,\theta)  := \int_{\R^d} ( (1 + \alpha_{\theta} \|\vect(y y^\ast)\|_{S_{\theta}})^p - 1) \nu_L(\diff y) + 2 p \max\{\mathfrak{R}(\sigma(B_{\theta}))\},
\end{equation}
$\alpha_{\theta} = \|S_{\theta}\|_2^2 \|S_{\theta}^{-1}\|_2^2 K_{2,B_{\theta}} \|A_{\theta} \otimes A_{\theta} \|_{S_{\theta}}$ with $K_{2,B_{\theta}}$ as in \eqref{eq:defK2Y}, $\|\vect(y y^\ast)\|_{S_{\theta}}$ as in \eqref{eq:defxYn} and $S_{\theta}$ as in \ref{as:par2}.

\end{enumerate}
\end{assumptionf}

\begin{assumptiong}[MUCOGARCH process at $\theta_0$]
\mbox{}
\begin{enumerate}[label=$(g.\arabic*)$]

\item \label{as:irre} The MUCOGARCH volatility process $Y$ is stationary, $\mu$-irreducible with the support of $\mu$ having non-empty interior and aperiodic.

\item\label{as:par8} $m(p,\theta_0) < 0$ for some $p > 4$.

%

\end{enumerate}

\end{assumptiong}





Assumption~\ref{as:par6}-\ref{as:par5} collect the needed identifiability assumptions from Section~\ref{se:identif}. Assumption~\ref{as:par7} is a sufficient condition under which we have uniqueness of the stationary distribution of $Y$ and geometric ergodicity (Section~\ref{se:mixing}). For the asymptotic results of the GMM estimator in Theorems~\ref{th:conGMM} and \ref{th:NormGMM}, we need to ensure $\E \|Y_0\|^p < \infty$ for appropriate $p > 1$, and this would require checking Assumptions \ref{eq:thcii} or \ref{eq:thciii} with $p > 1$. However, this imposes strong conditions on the L\'evy process \cite[Remark~4.4]{Stelzer17}. Instead, we require diagonalizability of the matrix $B_\theta$ (\ref{as:par2}), which is not a very restrictive assumption, and check \eqref{eq:con44} to ensure $\E \|Y_0\|^p < \infty$ \cite[Theorem~4.5]{Stelzer10}, which is less restrictive.


In view of the above assumptions and the results of Sections~\ref{se:gmm}-\ref{se:identif} we have the following consistency result.

\begin{corollary}[Consistency of the GMM estimator - $L$ has paths of infinite variation]\label{co:cons2} 
Suppose that assumptions \textbf{a}, \textbf{b}, \ref{as:comp},\ref{as:par6}-\ref{as:par10} and \ref{as:irre} hold. Then the GMM estimator defined in \eqref{eq:defGMM} is weakly consistent.
\end{corollary}

If the paths of the driving L\'evy process are of finite variation, we can relax even more the conditions from Corollary~\ref{co:cons2}. Before we state this result, we give the definition of asymptotic second-order stationarity which will be used in its proof. A stochastic process $X \in \S_d$ is said to be asymptotically second-order stationary with mean $\mu \in \R^{d^2}$, variance $\Sigma \in \S_{d^2}^+$ and autocovariance function $f: \R^+ \mapsto M_{d^2}(\R)$ if it has finite second moments and
\begin{equation*}
    \begin{split}
& \lim _{t \rightarrow \infty} \E\left(X_{t}\right)=\mu, \quad\quad \lim _{t \rightarrow \infty} \var\left(\vect\left(X_{t}\right)\right)=\Sigma \\
        & \lim _{t \rightarrow \infty} \sup _{h \in \mathbb{R}^{+}}\left\{\left\|\cov\left(\vect\left(X_{t+h}\right), \vect\left(X_{t}\right)\right)-f(h)\right\|\right\}=0.
    \end{split}
\end{equation*}

\begin{corollary}[Consistency of the GMM estimator - $L$ has paths of finite variation]\label{co:cons2fv}
Suppose that assumptions \textbf{a}, \textbf{b}, \ref{as:comp}, \ref{as:par6}, \ref{as:par1}, \ref{as:par3}-\ref{as:par7}, \ref{as:irre} hold and that $L$ has paths of finite variation. Then, the GMM estimator defined in \eqref{eq:defGMM} is weakly consistent.
\end{corollary}

\begin{proof}
Let $D \in \mathbb{S}_{d}^{+}$ be a constant matrix, and consider a MUCOGARCH process $(Y_t)_{t \in \R^+}$ solving \eqref{eq:dYt} have starting value $D$. Then, a combination of Assumptions \ref{as:varL1_id}, \ref{as:Equadquad} with the fact that the starting value $D$ is non-random and the hypothesis imposed on the matrices $B_\theta, \calb_\theta, \calc_\theta$ allow us to apply Theorem~4.20(ii) in \cite{Stelzer10} to conclude that the process $(Y_t)_{t \in \R^+}$ is asymptotically second-order stationary. Additionally, Theorem~\ref{th:331Vest}(i) ensures that the process $(Y_t)_{t \in \R^+}$ has a unique stationary distribution, is geometrically ergodic and its stationary distribution has finite first moment, i.e., $\E \|Y_0\| < \infty$. Since $Y_t \in \S_d^+$, and $\text{tr}(Y^* Y)$ (with $\text{tr}$ denoting the usual trace functional) defines a scalar product on $\S_d$ via $\text{tr}(Y_t^* Y_t) =  \vect(Y_t^*) \vect(Y_t)$ it follows that
\begin{equation}\label{eq:naosei}
\begin{split}
\E \|Y_t\|_2^2  = & \text{tr}(Y_t^* Y_t) = \vect(Y_t^*) \vect(Y_t) = \sum_{i,j}^d \E Y_{t,ij}^2 =  \sum_{i,j}^d \var(Y_{t,ij})  + \sum_{i,j}^d  (\E Y_{t,ij})^2\\
& \text{tr} (\var( Y_t)) +  \|\E (Y_t) \|_2^2, \quad t > 0.
\end{split}
\end{equation}
Since both maps $t \mapsto \E \| Y_t \|$ and $t \mapsto \var(Y_t) $ are continuous (\cite[eqs. (4.7) and (4.16)]{Stelzer10}), it follows from  \eqref{eq:naosei} that $\limsup_{t \geq 0} \E \| Y_t \|^2 < \infty$. Since Theorem~\ref{th:331Vest}(i) implies convergence of the transition probabilities in total variation, which in turns implies weak convergence (e.g. \cite[Exercise 13.2.2]{Klenke13Prob}), we have that $Y_t \std Y_0$ as $\tto$, with $Y_0$ being the stationary version of $Y$. Hence, we can use the continuous mapping theorem and \cite[Theorem~25.11]{billingsley} to conclude that $\E \|Y_0\|^2 < \infty$. Finally, the result follows by an application of Lemma~\ref{le:momG1}, Theorem~\ref{th:conGMM},  Corollary~\ref{co:expBmix} and Remark~\ref{re:new}.
\end{proof}

Recall that for the asymptotic normality result, we need to ensure that the stationarity distribution of the MUCOGARCH volatility process has more than $4$ moments (cf. \ref{as:momANC}). This is summarized in the next corollary.

\begin{corollary}[Asymptotic normality of the GMM estimator]\label{co:anorm2}
If Assumptions~\textbf{a}, \textbf{b}, \ref{as:comp}, \ref{as:theInt}, \ref{as:diffABC}, \textbf{f} and \textbf{g} hold, then the GMM estimator defined in \eqref{eq:defGMM} is asymptotically normal with covariance matrix as in \eqref{eq:acmGMM}.
\end{corollary}
\begin{proof}
By the same arguments of \cite[Theorem~4.5]{Stelzer10} combined with \cite[Proposition~4.1]{LM05OU} and \ref{as:par8}, it follows that $\E \|Y_0\|^p < \infty$ for some $p > 4$. The rest of the proof is just an application of Theorem~\ref{th:NormGMM}. 
\end{proof}

\begin{remark}
The advantage of Corollaries~\ref{co:cons2}-\ref{co:anorm2} is that Assumption~\textbf{f} can be checked numerically and Assumption~\textbf{g} holds true if e.g., the L\'evy process $L$ is a compound Poisson process with jump distribution having a density which is strictly positive around zero (see Theorem~\ref{th:Th51JS}).

If \ref{as:irre} holds, the stationary distribution of $Y$ is automatically a maximum irreducibility measure. All maximal irreducibility measures are equivalent and thus the support of the stationary distribution has a support which has non-empty interior. The latter in turn implies that the variance has to be an invertible operator (non-invertibility is equivalent to the distribution being  concentrated on a proper linear subspace) which is \ref{as:par4} for $\theta_0$.
\end{remark}

In the next section, we investigate the finite sample performance of the estimators in a simulation study.

\section{Simulation study}\label{s:sim}

To assess the performance of the GMM estimator, we will focus on the MUCOGARCH model in dimension $d = 2$. We fix $L_t = L^\disc_t + \sqrt{\sigma_W}W_t$ for $t \in \R^+$ where $L^\disc$ is a bivariate compound Poisson process (CPP), $W$ is a standard bivariate Brownian motion, independent of $L^\disc$ and $\sigma_W \geq 0$ is fixed. We choose $L^\disc$ as a CPP, since it allows to simulate the MUCOGARCH volatility process $V$ exactly. Thus, we only need to approximate the Brownian part of the (log) price process $G$ in \eqref{eq:def:Pt}, which is done by an Euler scheme. Setting $L^\disc$ as a CPP is not a very crucial restriction, since for an infinite activity L\'evy process one would need to approximate it using only finitely many jumps. For example by using a CPP for the big jumps component of $L^\disc$ and an appropriate Brownian motion for its small jumps component (see \cite{cohen2007gaussian}). In applications, a CPP has also been used in combination with the univariate COGARCH(1,1) process for modeling high frequency data (see \cite{Muller10}). The jump distribution of $L^\disc$ is chosen as $N(0, 1/4 I_2)$ and the jump rate is $4$, so that $\var(L_1) = 2 I_2$ and 
$$
\E [ \vect(\qvl^\disc), \vect(\qvl^\disc)^{\ast} ]_1^\disc =  1/4 ( I_{4} + K_{2} + \vect(I_2) \vect(I_2)^{\ast} ).
$$
In this case, the chosen L\'evy process $L$ satisfies Assumptions \textbf{a} from Section~\ref{se:mFor} (with $\sigma_L = 1$ and $\sigma_W > 0$). Based on the identification Lemma~\ref{le:IdABFrCB}, we assume that the model is parameterized with $\theta= (\theta^{(1)},\dots,\theta^{(11)})$, and the matrices $A_\theta, B_\theta$ and $C_\theta$ are defined as:
\begin{equation*}
A_\theta = \begin{pmatrix}
\theta^{(1)} & \theta^{(2)} \\
\theta^{(3)} & \theta^{(4)} \\
\end{pmatrix}, \quad 
B_\theta = \begin{pmatrix}
\theta^{(5)} & \theta^{(6)} \\
\theta^{(6)} & \theta^{(7)} \\
\end{pmatrix} \quad \text{and}  \quad
C_\theta = \begin{pmatrix}
\theta^{(8)} & \theta^{(9)} \\
\theta^{(9)} & \theta^{(10)} \\
\end{pmatrix},
\end{equation*}
with  $ \theta^{(2)} > 0$ and $ \theta^{(2)} \not=  \theta^{(3)}$. Thus, Assumption~\ref{as:diffABC} and \ref{as:par3} are automatically satisfied.
The data used for estimation is a sample of the log-price process $\bsg = (\bsg_i)_{i=1}^n$ as defined in \eqref{eq:def:Gn} with true parameter value $\theta_0 \in \Theta \subset \R^{11}$ observed on a fixed grid of size $\Delta = 0.1$ (the grid size for the Euler approximation of the Gaussian part is $0.01$).

We experiment with two different settings, namely:

\begin{example}\label{ex:1} We fix $\sigma_W = 1$,
\begin{equation}\label{eq:theta0}
A_{\theta_0} = \begin{pmatrix}
0.85 & 0.10 \\
-0.10 & 0.75 \\
\end{pmatrix}, \quad 
B_{\theta_0} = \begin{pmatrix}
-2.43 & 0.05 \\
0.05 & -2.42 \\
\end{pmatrix} \quad \text{and}  \quad
C_{\theta_0} = \begin{pmatrix}
1 & 0.5 \\
0.5 & 1.5 \\
\end{pmatrix}.
\end{equation}
\end{example}
\begin{example}\label{ex:2} We fix $\sigma_W = 0,$
$A_{\theta_0}$ and $C_{\theta_0}$ are as in Example~\ref{ex:1} and 
\begin{equation}\label{eq:theta0}
B_{\theta_0} = \frac{1}{4}\begin{pmatrix}
-2.43 & 0.05 \\
0.05 & -2.42 \\
\end{pmatrix}. 
\end{equation}
\end{example}


For the chosen L\'evy process here, Assumption~\ref{as:irre} is satisfied. In Example~\ref{ex:1}, $\theta_0$ is chosen in such a way that the asymptotic normality of $\hat{\theta}_n$ can be verified. 
Then, in Example~\ref{ex:2} we rescale $B_{\theta_0}$ from Example~\ref{ex:1} in such a way that our sufficient conditions for weak consistency are satisfied, but our sufficient conditions for asymptotic normality in Corollary~\ref{co:anorm2} are not satisfied.


Due to the identifiability Lemma~\ref{le:IdABFrCB} we need to choose $r \geq 2$. For comparison purposes, we perform the estimation for maximum lags $r \in \{2,5,10\}$ and sample sizes $n \in   \{1\,000,  10\,000, 100\,000\}$. The computations are performed with the \texttt{optim} routine in combination with the \texttt{Nelder-Mead} algorithm in R (\cite{Rsoftware}). Initial values for the estimation were found by the \texttt{DEoptim} routine on a neighborhood around the true parameter $\theta_0$.  We only consider estimators based on the identity matrix for the weight matrix $\Omega$ in \eqref{eq:defGMM}. The results are based on $500$ independent samples of MUCOGARCH returns. 


In the following we report the finite sample results of the GMM for Examples~\ref{ex:1} and \ref{ex:2}.

\subsection{Simulation results for Example~\ref{ex:1}}

We can check numerically that the matrices $A_{\theta_0}, B_{\theta_0}$ and $A_{\theta_0}$ are such that Assumptions~\textbf{b} and \ref{as:par1}-\ref{as:par5} hold. Additionally, the eigenvalues of the matrix
$B_{\theta_0} + B_{\theta_0}^* + \sigma_L A_{\theta_0}^* A_{\theta_0}$ are  $-4.067$ and $-4.328$, so it is negative definite and Assumptions~\ref{as:par7} holds. For our choice of $\theta_0$ we have that $B_{\theta_0}$ is diagonalizable with $B_{\theta_0} = S_{\theta_0}D_{\theta_0}S_{\theta_0}^{-1}$, where
\begin{equation*}
S_{\theta_0} = \begin{pmatrix}
-0.671 & -0.741 \\
-0.741 & 0.671 \\
\end{pmatrix} \quad \text{and}  \quad
D_{\theta_0} = \begin{pmatrix}
-2.375 & 0 \\
0 & -2.475 \\
\end{pmatrix}.
\end{equation*}
In addition, for $p = 4.001$,
\begin{equation}\label{eq:conprac}
\int_{\R^2} ( (1 + \alpha_{\theta_0} \|\vect(y y^\ast)\|_{S_{\theta_0}})^p - 1) \nu_L(\diff y) + 2 p \max\{\mathfrak{R}(\sigma(B_{\theta_0}))\} = -0.024 < 0,
\end{equation}
so \ref{as:par8} is also valid. Therefore, all assumptions for applying Corollary~\ref{co:anorm2} can be verified, which imply assumption~\textbf{e}, and ensure asymptotic normality. We also note that the chosen parameters are very close to not satisfying Assumption~\eqref{eq:conprac}. 


We investigate the behavior of the bias and standard deviation in Figure~\ref{fig:bias_and_std}, where we excluded those paths for which the algorithm did not converge successfully (around $10$ percent of the paths of length $n = 1\,000$ and less than $3$ percent for larger $n$). Figure~\ref{fig:bias_and_std} show the estimated absolute values of the bias and standard deviation for different lags $r$ and varying $n$. As expected, they decay when $n$ increases. Additionally, the results favor the choice of maximum lag $r = 10$, which is already expected since using more lags of the autocovariance function usually helps to give a better fit. It is also worth noting that the estimation of the parameters in the matrix $B_{\theta_0}$ is more difficult than the other parameters, specially for $n \in \{1\,000,10\,000\}$.

Figures~\ref{fig:qqplots1} and \ref{fig:qqplots2} assess asymptotic normality though normal QQ-plots. Based on the previous findings we fix $r = 10$, since it gave the best results. This might have to do with the fact that using just a few lags for the autocovariance function ($r=2$ or $r=5$) are not sufficient for a good fit. Here we do not exclude those paths for which the algorithm did not converge (these are denoted by large red points in the normal QQ-plots in Figures~\ref{fig:qqplots1} and \ref{fig:qqplots2}). These plots are clearly in line with the asymptotic normality of the estimators. It is worth noting that the tails corresponding to the estimates of $B_{\theta_0}$ deviate from the ones of a normal distribution for values of $n \in \{1\,000,10\,000\}$, but they get closer to a normal distribution for $n = 100\,000$. The tails of the plots for $A_{(2,1),\hat{\theta}_n}$ in Figure~\ref{fig:qqplots1} is not close to a normal (although the plots show its convergence). This is maybe due to identifiability condition in Lemma~\ref{le:IdABFrCB} which requires $A_{(2,1),\theta} > 0$ but $A_{(2,1),\theta_0} = 0.1$ is very close to the boundary. For $n = 1\,000$, there are very large negative outliers for the estimates of $B_{\theta_0}$, which affects the bias substantially.

 \begin{minipage}{\linewidth}
 \makebox[\linewidth]{
   \includegraphics[width = 17cm, height=15cm]{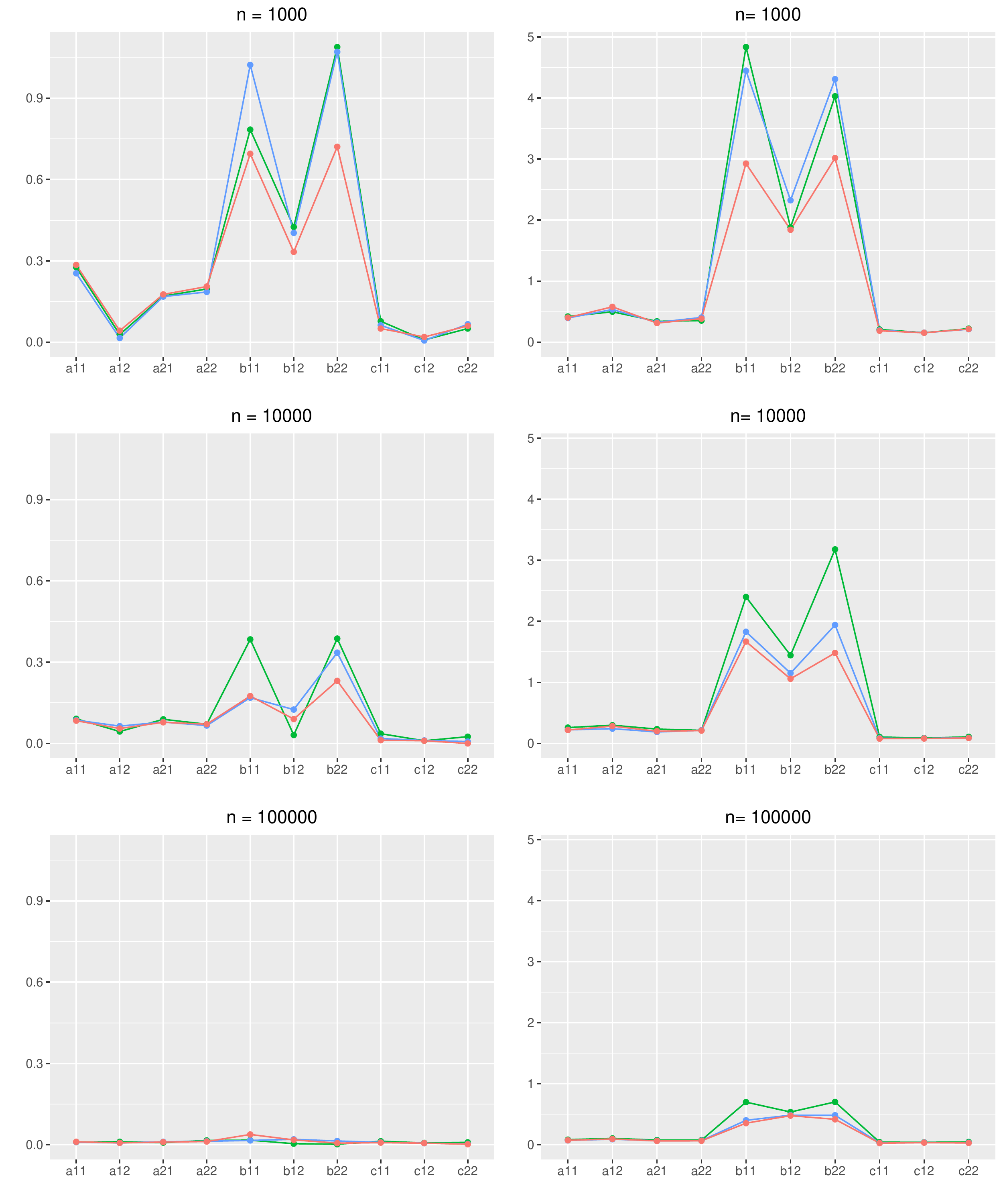}}
 \captionof{figure}{Example 1: Estimated absolute bias (lhs) and standard deviation (rhs) of $\hat{\theta}_{n,r}$. The colors green, blue and red correspond to $r=2,5$ and $10$, respectively.}\label{fig:bias_and_std}
 \end{minipage}


 \begin{minipage}{\linewidth}
 \makebox[\linewidth]{
   \includegraphics[width = 17cm, height=20cm]{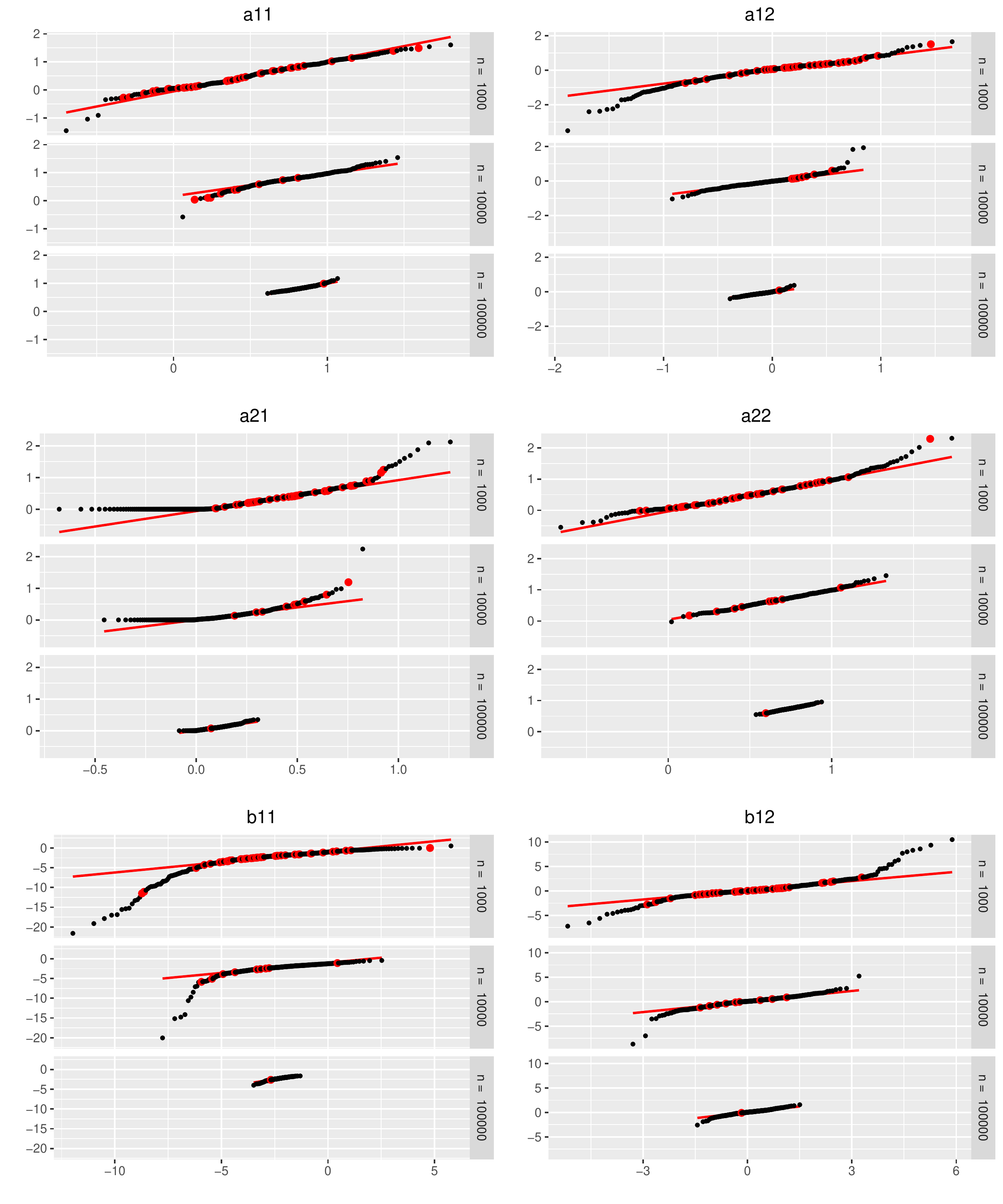}}
 \captionof{figure}{Example 1: Normal QQ-plots of $\hat{\theta}_{n,10}$ for $\theta_0$ as in \eqref{eq:theta0}. The red dots are values for which the algorithm did not converge.}\label{fig:qqplots1}
 \end{minipage}

 \begin{minipage}{\linewidth}
 \makebox[\linewidth]{
   \includegraphics[width = 17cm, height=20cm]{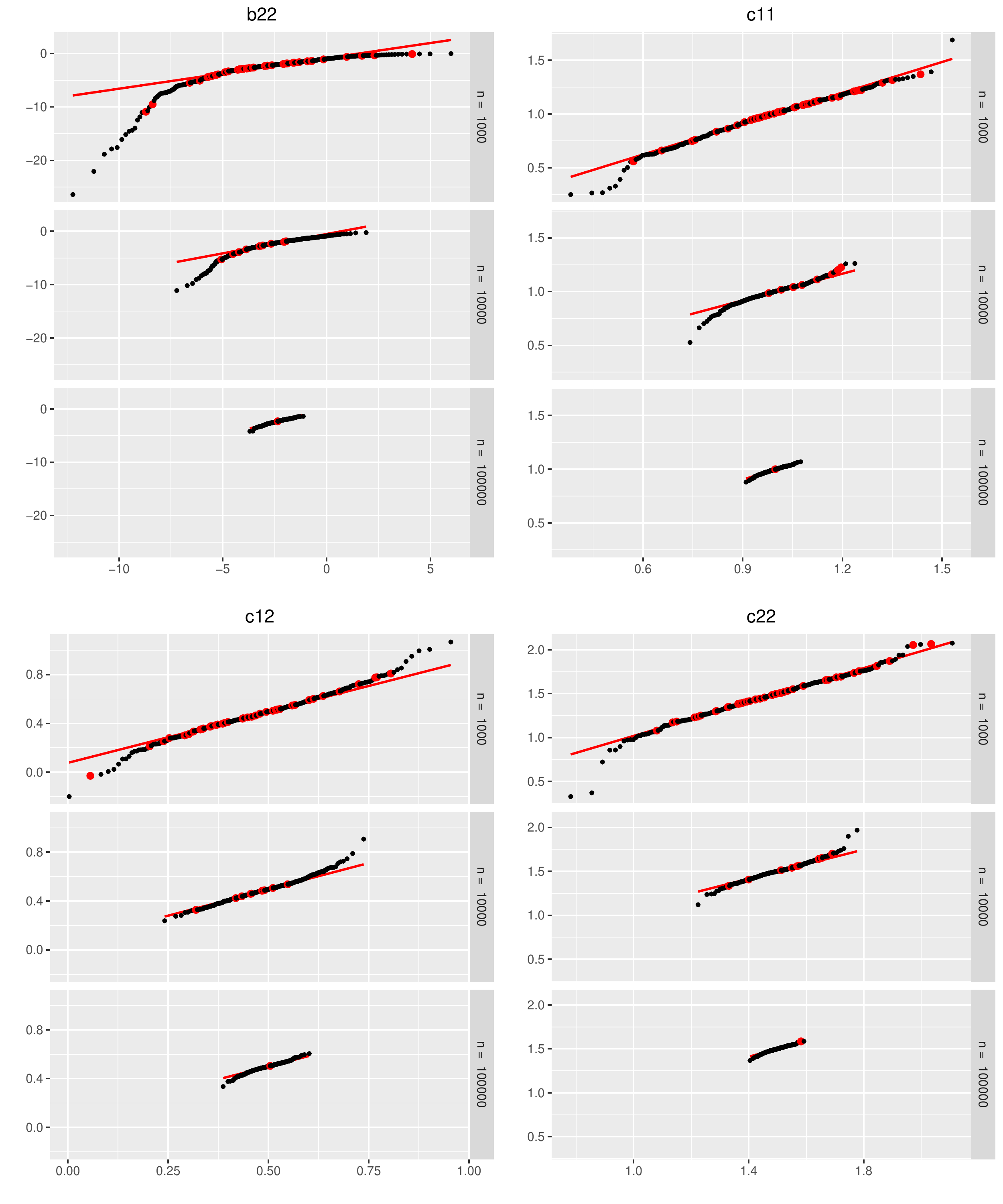}}
 \captionof{figure}{Example 1: Normal QQ-plots of $\hat{\theta}_{n,10}$ for $\theta_0$ as in \eqref{eq:theta0}. The red dots are values for which the algorithm did not converge.}\label{fig:qqplots2}
 \end{minipage}

 
  \begin{minipage}{\linewidth}
 \makebox[\linewidth]{
   \includegraphics[width = 17cm, height=15cm]{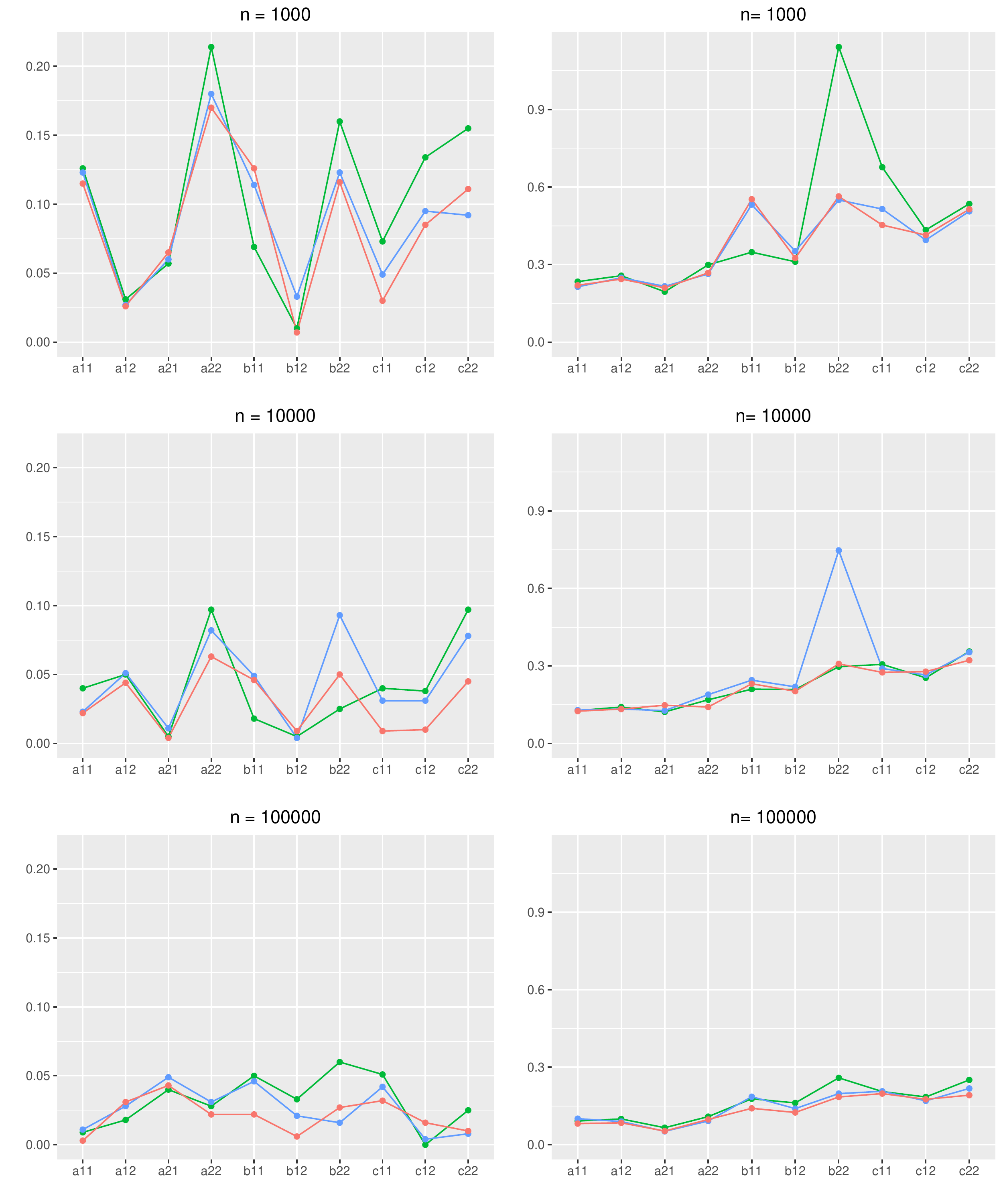}}
 \captionof{figure}{Example 2: Estimated bias and std of $\hat{\theta}_{n,r}$. The colors green, blue and red correspond to $r=2,5$ and $10$, respectively.}\label{fig:bias_2}
 \end{minipage}
 


 \begin{minipage}{\linewidth}
 \makebox[\linewidth]{
   \includegraphics[width = 17cm, height=20cm]{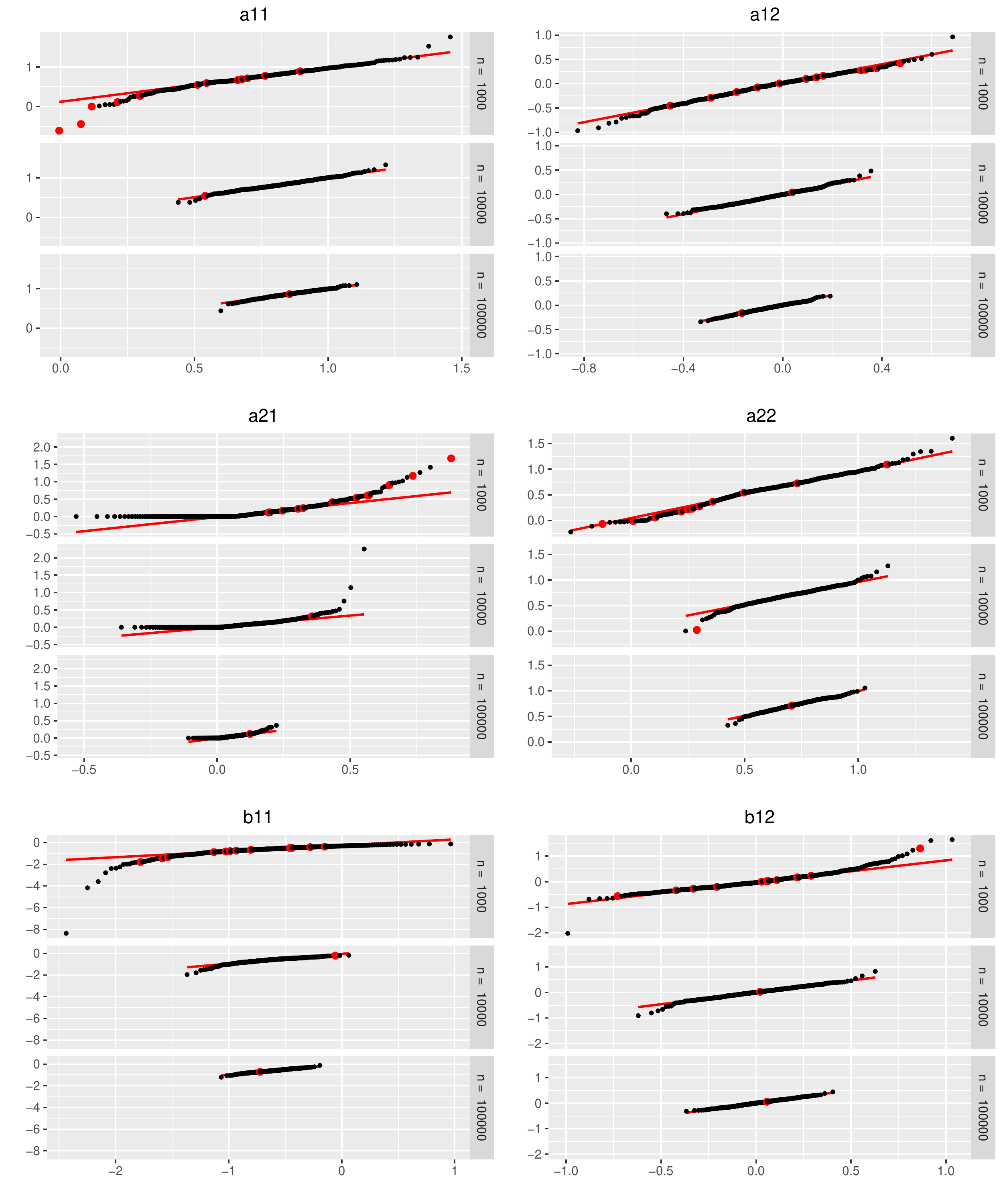}}
 \captionof{figure}{Example 2: Normal QQ-plots of $\hat{\theta}_{n,10}$ for $\theta_0$ as in \eqref{eq:theta0}. The red dots are values for which the algorithm did not converge.}
 \end{minipage}

 \begin{minipage}{\linewidth}
 \makebox[\linewidth]{
   \includegraphics[width = 17cm, height=20cm]{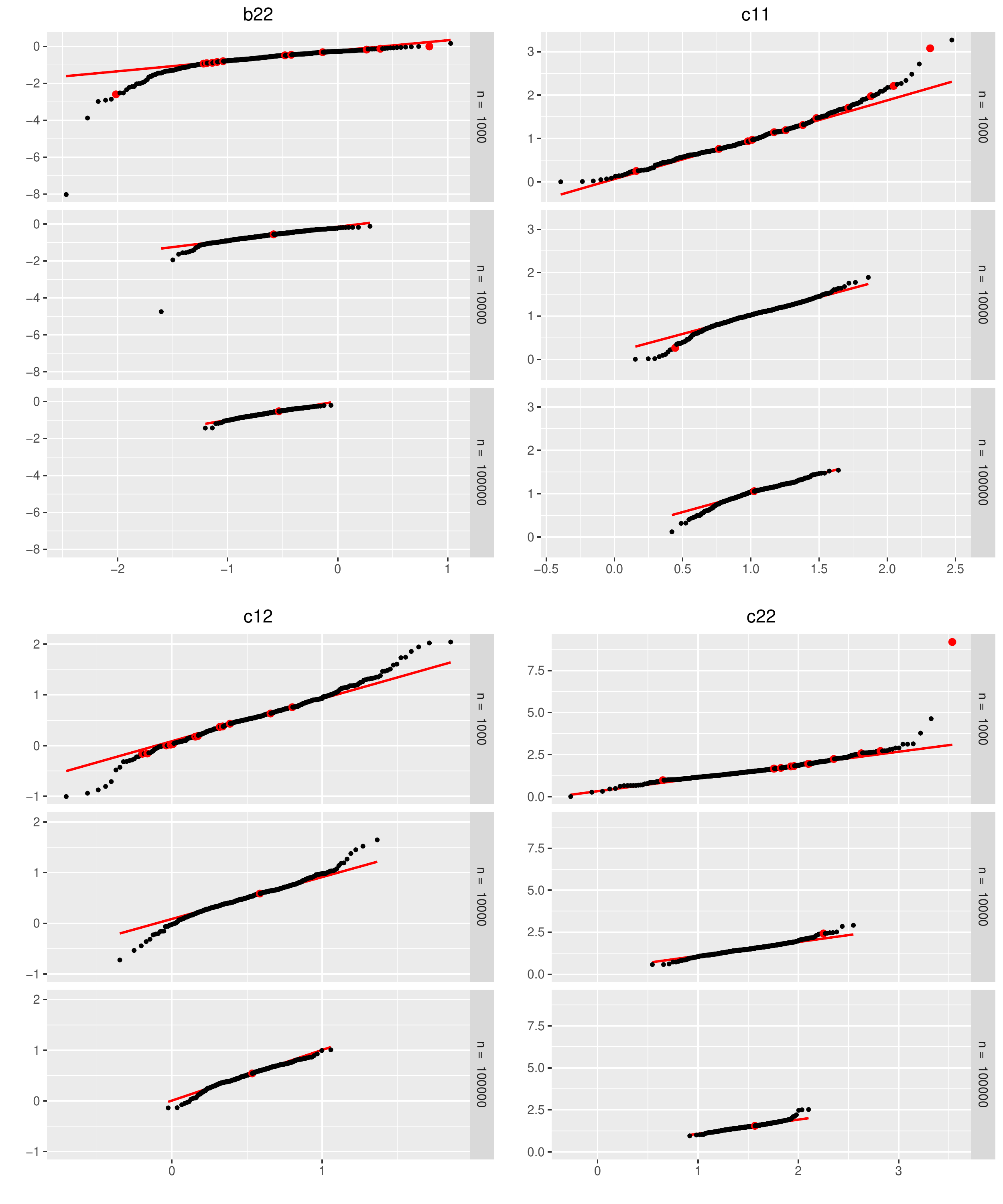}}
 \captionof{figure}{Example 2: Normal QQ-plots of $\hat{\theta}_{n,10}$ for $\theta_0$ as in \eqref{eq:theta0}. The red dots are values for which the algorithm did not converged.}\label{fig:qqplots2_2}
 \end{minipage}

\subsection{Simulation results for Example~\ref{ex:2}}
In this section we analyze the behavior of the GMM estimator when the consistency conditions are valid, but we cannot check the conditions for asymptotic normality. Here, we have
$\sigma(B_{\theta_0}/4 + B_{\theta_0}^*/4 +  \sigma_L A_{\theta_0}^* A_{\theta_0})  = \{-0.594, -0.619\} \in (-\infty,0) + i\R$. Thus, Corollary~\ref{co:cons2fv} applies and gives weak consistency of the GMM estimator. On the other hand, for $p = 4.001$ the integral in \eqref{eq:conprac} is $14.22 > 0$, and thus, we cannot apply Corollary~\ref{co:anorm2} to ensure asymptotic normality.

The results for Example~\ref{ex:2} are given in Figures~\ref{fig:bias_2}-\ref{fig:qqplots2_2}. The estimation of the entries of $B_{\theta_0}$ does not seem to be substantially more difficult than the entries of $A_{\theta_0}$ and $C_{\theta_0}$, as observed in the previous example. Also, the estimated bias and std decreases in general as $n$ grows, showing consistency of the estimators. Also,  the convergence rate seems slow and, therefore, probably smaller than $n^{1/2}$ (the asymptotic normality rate from Theorem~\ref{th:NormGMM}). The QQ-plots for the estimation of the parameters $A_{(2,1)}$, $C_{(1,1)}$ and $C_{(2,1)}$ also show some deviation from the normal distribution.




\section{Real data analysis}\label{se:realdata}

In this Section, we fit the MUCOGARCH model to 5 minutes log-returns of stock prices corresponding to the SAP SE and Siemens AG companies (the data was obtained from the Refinitiv EIKON system). For both datasets we have excluded overnight returns. The resulting bivariate dataset has a total length of 12 135 (from 30-jun-2020 to 15-dez-2020) and is shown in Figure \ref{fig:data}.

 \begin{minipage}{\linewidth}
 \makebox[\linewidth]{
   \includegraphics[width = 15.5cm, height=11cm]{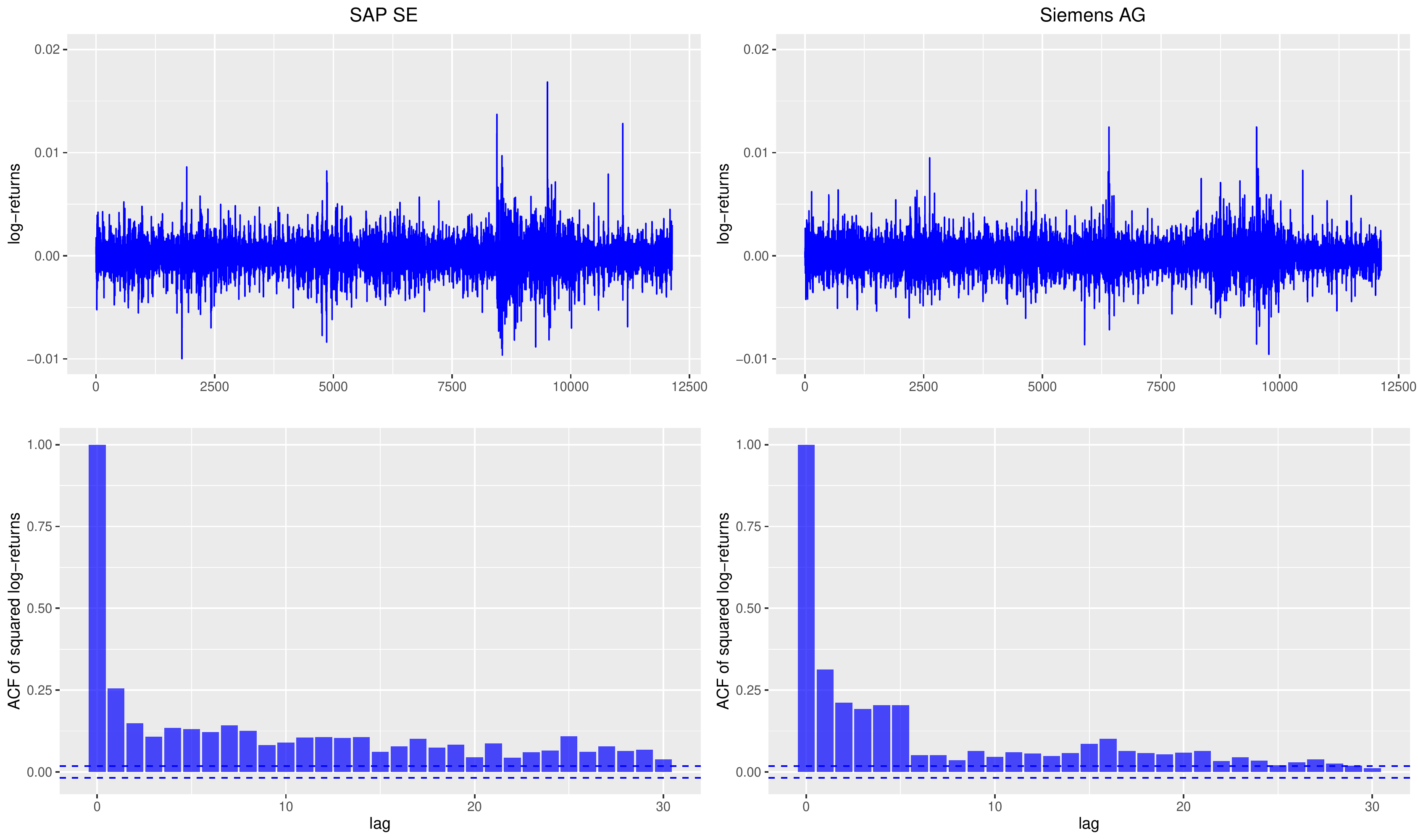}}
 \captionof{figure}{Log-returns (top) and  sample  acf  of  squared log-returns  (bottom)  for SAP SE (left) and Siemens AG (right) exchange rates.}\label{fig:data}
 \end{minipage}
 
  Based on the sample autocorrelation function of the squared log-returns we decided to use $15$ lags to estimate a bivariate MUCOGARCH model. Before moving to the estimation step, we re-scaled the data multiplying it by $1 000$ for numerical reasons. Since the GMM algorithm from \eqref{eq:defGMM} requires a starting value, we first fit a one-dimensional COGARCH to each dataset with the algorithm described in \cite[Algorithm 1]{Haug07} and used them to construct the corresponding 2-dimensional MUCOGARCH describing two independent log-return price process. This will be a MUCOGARCH process whose parameters $A,B,C$ are diagonal matrices and the driving L\'{e}vy process has independent components \cite[Example 4.2]{Stelzer10}. To construct a second step GMM estimator, we use the estimated parameters from the first step to estimate a weight matrix to replace $\Omega$ in \eqref{eq:defGMM}.
  Indeed, as explained in \cite[Section 1.3.3]{matyas99GMM}, and given a consistent estimator $\hat{\Sigma}_{\theta_0}$ of $\Sigma_{\theta_0}$, an asymptotically more efficient estimator can be constructed as 
  \begin{equation}\label{eq:defGMM2}
\hat{\theta}^{(2)}_n = \argminA_{\theta \in \Theta} \Big\{ (\hat{k}_{n,r} - k_{\theta,r})^T \hat{\Sigma}_{\theta_0} (\hat{k}_{n,r} - k_{\theta,r}) \Big\}.
\end{equation}

Given a bivariate dataset of log-price returns of size $n$ following a MUCOGARCH model, we can see from \eqref{eq:defSigma} that a natural estimator for $\Sigma_{\theta_0}$ is
\begin{equation}\label{eq:sighat}
\begin{split}
\hat{\Sigma}_{n,M} & = \frac{1}{n} \sum_{t=1}^n (D_t -\hat{k}_{n,r})(D_t -\hat{k}_{n,r})^* + \\
& \frac{1}{n} \sum_{t=1}^n \sum_{i=1}^M \big\{ (D_t -\hat{k}_{n,r})(D_{t+i} -\hat{k}_{n,r})^* + (D_{t+i} -\hat{k}_{n,r})^*(D_t -\hat{k}_{n,r}) \big\}, \quad M \in \mathbb{N},
 \end{split}
\end{equation}  
where we truncated the infinity sum in $\eqref{eq:defSigma}$. The above estimator is a symmetric matrix, regardless of the values chosen for $M$ and $n$. As done in the proof of Lemma~\ref{le:ConsM}, sufficient conditions for it to be consistent are assumptions \textbf{a}, \textbf{b}, \textbf{c} and \ref{as:mixing}. On the other hand, we cannot guarantee that it is going to be positive semidefinite, which is a condition required to use it as a weight matrix in the estimator from \eqref{eq:defGMM2}. One way around it is to use the simpler estimator (as in \cite[Section 4.1]{stelzer2015moment}) of $\Sigma_{\theta_0}$, which ignores the second summation on the rhs of \eqref{eq:sighat}, namely, 
\begin{equation}\label{eq:sighatBasicCOV}
\hat{\Sigma}_{n,M}^{(\text{basicCOV})} = \frac{1}{n} \sum_{t=1}^n (D_t -\hat{k}_{n,r})(D_t -\hat{k}_{n,r})^*, \quad M \in \mathbb{N},
\end{equation} 
The resulting estimators when applying the 2 step GMM with \eqref{eq:sighatBasicCOV} for the bivariate dataset of 5min log-returns are given by:
\begin{equation}\label{eq:2gmmrealdata}
A = \begin{pmatrix}
0.442 & 0.259 \\
0.054 & 0.054 \\
\end{pmatrix}, \quad 
B = \begin{pmatrix}
-0.146 & -0.014 \\
-0.014 & -0.080 \\
\end{pmatrix} \quad \text{and}  \quad
C = \begin{pmatrix}
0.257 & -0.134 \\
-0.134 & 0.070 \\
\end{pmatrix}.
\end{equation}
The parameters of the Levy process were chosen to be the same from example 1 as they allow us to check consistency and asymptotic normality conditions. According to Corollary~\ref{co:cons2fv}, the estimated parameters given in \eqref{eq:2gmmrealdata} describe a MUCOGARCH model for which the GMM estimators is consistent. For the asymptotic normality, we need to verify assumption (f.8), which requires computing $m(p,\theta_0)$ with $\theta_0$ replaced by a vector formed with the entries of the matrices in \eqref{eq:2gmmrealdata}. This computation results in 2.98, and therefore (f.8) is violated, and we cannot ensure asymptotic normality. 

To assess performance of the GMM estimator used to estimate the MUCOGARCH model for real data,  we perform a simulation study using the values from \eqref{eq:2gmmrealdata} as $\theta_0$. The results are reported in Table~\ref{tb:simReal} and referred by basicCOV. We also compare it with the first step GMM and two other 2 step GMM estimators: one using $\hat{\Sigma}_{n,M}$ from \eqref{eq:sighat}  with $M = 10$ as a weight matrix and another which uses a weighing diagonal matrix formed with the diagonal of $\hat{\Sigma}_{n,M}$. These are referred in Table~\ref{tb:simReal} as fullCOV and diagCOV, respectively. Using fullCOV did not improved the estimates when compared with basicCOV and diagCOV, and that might be to do with the fact that inverting the estimated covariance matrix gave several warnings during the estimation. The estimator based on diagCOV was somehow similar to the 1 step GMM, giving smaller bias and std for the parameters $\theta^{(6)}-\theta^{(8)}$, which correspond to the matrix $B$. The fact that the 2 step GMM did not improved the 1step GMM results, might be because the sample size $n$ used to estimate the covariance matrix is too small. 

\begin{table}[t]
\centering
\begin{small}
\begin{tabular}{rrrrrrr}
\multicolumn{6}{c}{Estimated Bias - GMM} \\
 \hline
& True Value & 1 step & 2 step (BasicCOV) & 2 step ($M=10$) & 2 step (diagCOV) \\ 
 \hline
$\theta^{(1)}$ & 0.442 &  -0.024&-0.094&-0.055&-0.073    \\ 
$\theta^{(2)}$ & 0.259  & 0.016&-0.026&0.040&-0.034   \\ 
$\theta^{(3)}$ & 0.054 & 0.086&0.120&0.121&0.100   \\ 
$\theta^{(4)}$ &0.194 &0.164&0.170&0.151&0.169   \\ 
$\theta^{(5)}$ & -0.146 &0.010&0.014&-0.043&0.031    \\ 
$\theta^{(6)}$ & -0.014 &-0.085&-0.076&-0.091&-0.072    \\ 
$\theta^{(7)}$ & -0.080 &-0.081&-0.077&-0.162&-0.058    \\ 
$\theta^{(8)}$ & 0.257 &0.091&0.129&0.248&0.083 \\ 
$\theta^{(9)}$ & -0.134& 0.147&0.148&0.176&0.148   \\ 
$\theta^{(10)}$ & 0.070 &0.136&0.157&0.315&0.137 \\ 
  \hline \\
  \multicolumn{6}{c}{Estimated Std - GMM} \\
  \hline
& True Value & 1 step & 2 step (BasicCOV) & 2 step ($M=10$) & 2 step (diagCOV) \\ 
 \hline
$\theta^{(1)}$ & 0.442   &  0.101&0.084&0.207&0.100   \\ 
$\theta^{(2)}$ & 0.259   &  0.123&0.112&0.261&0.112   \\ 
$\theta^{(3)}$ & 0.054 &  0.088&0.086&0.169&0.093   \\ 
$\theta^{(4)}$ &0.194 &  0.109&0.108&0.193&0.110    \\ 
$\theta^{(5)}$ & -0.146 &  0.080&0.091&0.402&0.067    \\ 
$\theta^{(6)}$ & -0.014 &  0.073&0.081&0.398&0.065   \\ 
$\theta^{(7)}$ & -0.080 &  0.072&0.085&0.650&0.062    \\ 
$\theta^{(8)}$ & 0.257 &  0.112&0.113&0.909&0.168    \\ 
$\theta^{(9)}$ & -0.134 &  0.093&0.101&0.645&0.139    \\ 
$\theta^{(10)}$ & 0.070 &  0.091&0.103&1.137&0.126   \\ 
  \hline
\end{tabular}
\end{small}
\caption{Estimated bias and std of $\hat{\theta}_{n,r}$ (1 step GMM) and $\hat{\theta}_{n,r}^{(2)}$, the 2 step GMM with 3 different weighting matrices, namely: basicCOV from \eqref{eq:sighatBasicCOV}, fullCOV from \eqref{eq:sighat} with $M = 10$ and diagCOV, a diagonal weighting matrix formed with the diagonal entries of \eqref{eq:sighatBasicCOV}.}\label{tb:simReal}
\end{table}

%
%
\subsubsection*{Acknowledgement}
Both authors take pleasure to thank Claudia Kl\"uppelberg for motivating this work and for helpful comments. Thiago do R\^ego Sousa gratefully acknowledges support from the National Council for Scientific and Technological Development (CNPq - Brazil) and the TUM Graduate School. He also thanks the Institute of Mathematical Finance at Ulm University for its hospitality and takes pleasure to thank Carlos Am\'{e}ndola for helpful discussions.
 
 \appendix 

 \section{Proofs}\label{se:proofs}


\subsection{Auxiliary results}

Several results related to the algebra of multivariate stochastic integrals will be used here, for which we refer to Lemma~2.1 in \cite{Behme2012}. Furthermore, we need the following.

\begin{fact}[{\cite[Lemma~6.9]{Stelzer10} with drift}]\label{fct:59drif} Assume that $(X_t)_{t \in \R^{+}}$ is an adapted cadlag $\md$-valued  process satisfying $\E (\|X_t \|) < \infty$ for all $t \in \R^{+}$, $t \mapsto \E (\|X_t \|)$ is locally bounded and $\ltpr$ is an $\R^d$-valued L\'{e}vy process of finite variation with $\E (\|L_1 \|) < \infty$. Then
$$
\E \int_0^{\Delta} X_{s-} \df L_s = \int_0^{\Delta} \E (X_{s-}) \E(L_1) \df s.
$$
\end{fact}

\begin{fact}\label{fct:compGen}
Let  $(A_t)_{t \in \R^{+}}$ in $M_{d^2}(\R)$, $(B_t)_{t \in \R^{+}}$ in $\monedtwo$ be adapted caglad processes satisfying $\E \|A_t \| \|B_t \| < \infty$ for all $t \in \R^{+}$, $t \mapsto \E \|A_t\| \|B_t \|$ is locally bounded and $\ltpr$ be an $\R^d$ valued L\'{evy} process satisfying Assumption 5.2 in \cite{Stelzer10}. Then,
$$
\E \int_0^t A_s \df  (\vect([L,L]_s)) B_s = (\sigma_W + \sigma_L) \int_0^t \E [ A_s \vect(I_{d}) B_s ] \df s.
$$
\end{fact}
\begin{proof}
First notice that $\vect([L,L]_s)$ is an $\R^{d^2}$-valued L\'{e}vy process with finite variation. Then it follows from Fact~\ref{fct:59drif} that
\begin{equation*}
\begin{split}
& \vect \bigg( \E \int_0^t A_s \df  (\vect([L,L]_s)) B_s \bigg) = \E \int_0^t (B_s \otimes A_s) \df  (\vect([L,L]_s) \\ 
& = \int_0^t \E (B_s \otimes A_s)   \E(\vect([L,L]_1)) \df s = (\sigma_W + \sigma_L) \int_0^t \E (B_s \otimes A_s)   \vect(I_{d}) \df s \\
& = (\sigma_W + \sigma_L)  \vect \bigg( \int_0^t \E ( A_s I_d B_s ) \df s \bigg),
\end{split}
\end{equation*}
so the result follows by an application of $\vect^{-1}$.
\end{proof}

\textbf{Proof of Lemma~\ref{le:momG1}}: It follows from \cite[Proposition~4.7]{Stelzer10} (with $k = p$) that $\E\|Y_t\|^{p} < \infty$ for all $t \in \R^+$ and $t \mapsto \E\|Y_t\|^{p}$ is locally bounded. Then an application of \cite[Theorem~66 of Ch. 5]{Protter90} together with the fact that $\E \|L_1\|^{2p} < \infty$ and the definition of $\vtpr$ in \eqref{eq:Vt} gives for all $t > 0$
\begin{equation*}
\begin{split}
\E \|G_t\|^{2p} & =\E \bigg\| \int_0^t V_{s-}^{1/2} \diff L_s \bigg\|^{2p} \leq c \int_0^t \E \| V_{s-}^{1/2}\|^{2p} \diff s \leq c \int_0^t \E \| C +Y_{s-}\|^{p} \diff s.
\end{split}
\end{equation*}

\begin{lemma}\label{le:covYGG}
Assume that Assumptions \ref{as:Elzero}-\ref{as:El14fin}, \textbf{b} and \ref{as:Y0E4} hold. Then,
\begin{equation}\label{eq:lecovYG}
\begin{split}
\quad\quad & \cov(\vect(Y_{\Delta}), \vect(\bm{G}_1 \bm{G}_1^{\ast})) = \cov(\vect(Y_{\Delta}), \vect(G_{\Delta} G_{\Delta}^{\ast})) \\
& =  \var (\vect (V_0))   (e^{\calb^{\ast}\Delta} - I_{d^2}) [ (\sigma_W + \sigma_L)(\calb^{\ast})^{-1} - 2((A \otimes A)^{\ast})^{-1} ], \quad \Delta \geq 0.
\end{split}
\end{equation}
\end{lemma}
\begin{proof}
Since \ref{as:El14fin}, \ref{as:Elzero} and \ref{as:YtsecSt} hold, we can apply Lemma~\ref{le:momG1} with $p=2$ to conclude that both $\|\vect(Y_\Delta)\|$ and $\|\bm{G}_1 \bm{G}_1^{\ast}\|$ are square integrable random variables and thus, the covariance at the left hand side of \eqref{eq:lecovYG} is finite.
Integration by parts formula \cite[p.~111]{Stelzer10} gives
\begin{equation}\label{eq:ipG1G1}
G_{\Delta} G_{\Delta}^{\ast} = \int_0^{\Delta} V_{s-}^{1/2} \df L_s G_{s-}^{\ast} + \int_0^{\Delta} G_{s-} \df L_s ^{\ast}  V_{s-}^{1/2} + \int_0^{\Delta}  V_{s-}^{1/2}  \df [L,L^{\ast}]_s  V_{s-}^{1/2} := A_{\Delta} + A_{\Delta}^{\ast} + C_{\Delta}.
\end{equation}
It follows from Lemma~\ref{le:momG1}(a) and (b) with $p = 2$ together with the Cauchy-Schwarz inequality that
\begin{equation}\label{eq:FinIntEVG}
\int_0^t \E ( \|V_{s-}^{1/2}\|_2 \|G_{s-}\|_2 )^2 \diff s \leq \int_0^t \big(\E \|V_{s-}\|_2^2 \big)^{1/2} \big(\E \|G_{s-}\|_2^4 \big)^{1/2} \diff s < \infty,
\end{equation}
where the finiteness is due to the fact that the integrand is locally bounded, and thus, also bounded on $(0,t)$. Therefore $(A_t)_{t \in \R^+}$ is a martingale and $A_t \in L^2$ for all $t \geq 0$.
Thus, the integration by parts formula, the formula $\df (\vect(A_s))^{\ast} = \df L_s^{*} (G_{s-}^{\ast} \otimes V_{s-}^{1/2})$ (Lemma~2.1(vi) in \cite{Behme2012}) imply
\begin{equation}\label{eq:EvecYA}
\begin{split}
\quad\,\, & \cov(\vect(Y_{\Delta}), \vect(A_{\Delta})) \\
& =  \E \big( \vect(Y_{\Delta}) ( \vect(A_{\Delta}))^{\ast} \big) - \E (\vect(Y_{\Delta}) )  \E ( \vect(A_{\Delta}))^{\ast} \\
& = \E \bigg(  \int_0^{\Delta} \vect(Y_{s-})  \df ( \vect(A_s))^{\ast} +  \int_0^{\Delta}   \df \vect(Y_{s}) (\vect(A_{s-}))^{\ast} + [\vect(Y),(\vect(A))^{\ast}]_{\Delta} \bigg) - 0 \\
& = 0 + \E \int_0^{\Delta} \df \vect(Y_{s}) (\vect(A_{s-}))^{\ast} + \E (  [\vect(Y),(\vect(A))^{\ast}]_{\Delta}).
\end{split}
\end{equation}
The first expectation in \eqref{eq:EvecYA} vanishes since 
\begin{equation*}
\begin{split}
& \int_0^{\Delta} \E \|\vect(Y_{s-})\|^2 \|G_{s-}\|^2 \|V_{s-}^{1/2}\|^2 \df s \\
& \leq \int_0^{\Delta} (\E \|\vect(Y_{s-})\|^4)^{1/2} ( \E \|G_{s-}\|^8 )^{1/4} ( \E \|V_{s-}^{1/2}\|^8 )^{1/4} \df s < \infty
\end{split}
\end{equation*}
by the generalized H\"older inequality with $(1/2 + 1/4 + 1/4 = 1)$ (see e.g. \cite[Theorem~2.1]{kufner1977function}), Lemma~\ref{le:momG1} and the fact that $\ltpr$ is an $L^2$-martingale. Let $\tilde{C}:= (B \otimes I + I \otimes B)$ and recall from p. 84 in \cite{Stelzer10} that 
\begin{equation}\label{eq:2}
\df \vect(Y_s) = \tilde{C} \vect(Y_{s-}) \df s + (A \otimes A) (\vshalf \otimes \vshalf) \df \vect ([L,L]_s^\disc).
\end{equation}
Using \eqref{eq:2}, the bilinearity of the quadratic covariation process, \cite[eq. (2.1)]{Behme2012}, Lemma~\ref{le:momG1}, Facts~\ref{fct:59drif}, \ref{as:lp_qv}, \eqref{eq:FinIntEVG} and the It\^o isometry we obtain
\beam
 &   & [\vect(Y),(\vect(A))^{\ast}]_{\Delta} \nonumber\\
& = & \bigg[  \int_0^{\cdot} \tilde{C} \vect(Y_{s-}) \df s + \int_0^{\cdot} (A \otimes A) (\vshalf \otimes \vshalf) \df \vect ([L,L]_s^\disc),  \int_0^{\cdot} \df L_s^{*} (G_{s-}^{\ast} \otimes V_{s-}^{1/2}) \bigg]_{\Delta} \nonumber\\
& =  & \int_0^{\Delta} (A \otimes A) (\vshalf \otimes \vshalf) \df[\vect ([L,L]^\disc), L^{\ast} ]_s (G_{s-}^{\ast} \otimes V_{s-}^{1/2})\label{eq:qvVYAa}.
\eeam
Recall that for arbitrary matrices $M \in M_{m,n}(\R)$ and  $N \in M_{k,l}(\R)$ it holds $\|A \otimes B\|_2 = \|A \|_2 \|B\|_2$ \cite[Fact~9.9.61]{Bernstein05}. This together with the H\"older inequality with (3/4 + 1/4 = 1) and Lemma~\ref{le:momG1} with $p=4$ gives
\begin{equation*}
\begin{split}
& \int_0^\Delta  \E \|\vshalf \otimes \vshalf\|_2 \| (G_{s-}^{\ast} \otimes V_{s-}^{1/2})  \|_2 \df s = \int_0^\Delta  \E \|\vshalf\|_2^3 \| G_{s-}^{\ast} \|_2 \df s \\
&  = \int_0^\Delta \E \|V_{s-}\|_2^{3/2} \| G_{s-}^{\ast} \|_2 \df s \leq  \int_0^\Delta (\E \|V_{s-}\|_2^2)^{3/4} (\E \|  G_{s-}^{\ast} \|_2^4)^{1/4} \df s < \infty.
\end{split}
\end{equation*}
Thus, applying expectations at both sides of \eqref{eq:qvVYAa} gives
\begin{equation}\label{eq:qvVYA}
E [\vect(Y),(\vect(A))^{\ast}]_{\Delta}  = 0.
\end{equation}

Let $l_{s}:= \E \vect(Y_{s})(\vect(A_{s}))^{\ast}$ and notice that it follows from Lemma~\ref{le:momG1} and the Cauchy-Schwarz inequality that $\E \|l_s\| < \infty$ and $s \mapsto \E \|l_s\|$ is locally bounded. Use \eqref{eq:2}, \eqref{eq:qvVYA}, the compensation formula, \cite[Proposition~7.1.9]{Bernstein05}, $\E \vect( V_s) \vect(A_{s-}) = l_s$, the It\^o isometry and \ref{as:varL1_id} to get
\begin{equation}\label{eq:ls_int}
\begin{split}
l_{\Delta} & =  \E \int_0^{\Delta} \df \vect(Y_{s}) (\vect(A_{s-}))^{\ast} \\
& = \E \int_{0}^{\Delta}  \big[ \tilde{C} \vect(Y_{s-}) \df s + (A \otimes A) (\vshalf \otimes \vshalf) \df (\vect ([L,L]_s^\disc) \big] (\vect(A_{s-}))^{\ast} \\
& = \tilde{C} \int_{0}^{\Delta} \E \vect(Y_{s-})(\vect(A_{s-}))^{\ast} \df s \\
& \quad\quad+ \sigma_L  \int_{0}^{\Delta} \E [ (A \otimes A ) (\vshalf \otimes \vshalf) \vect(I_d)(\vect(A_{s-}))^{\ast}] \df s \\
& =  (\tilde{C} +  \sigma_L(A \otimes A )  )\int_{0}^{\Delta} l_s \df s.
\end{split}
\end{equation}
Solving the matrix-valued integral equation in \eqref{eq:ls_int} and using that $A_0 = 0$ implies $l_{0} = 0$, gives $l_{s} = 0$ for all $s \geq 0$ (see \cite{Haug07}). Thus, it follows from \eqref{eq:EvecYA}-\eqref{eq:ls_int} that
\begin{equation}\label{eq:cvYAast}
\cov(\vect(Y_{\Delta}), \vect(A_{\Delta})) = 0,
\end{equation}
and, as a consequence $\cov(\vect(Y_{\Delta}), \vect(A_{\Delta}^\ast)) = 0$.
Let $\calv_{s-} := \vshalf \otimes \vshalf$. Then,
\beam
& & \vect(C_\Delta) = \int_0^{\Delta} \calv_{s-}  \df \vect([L,L^{\ast}]_s) = \int_0^{\Delta} \calv_{s-}  \df \vect([L,L^{\ast}]^\disc_s) \label{eq:IVssplit} \\
& & \quad \quad +  \sigma_W\int_0^{\Delta} (\vshalf \otimes \vshalf) \vect(I_d)  \df s =  \int_0^{\Delta} \calv_{s-}  \df \vect([L,L^{\ast}]^\disc_s) + \sigma_W\int_0^{\Delta} \vect(V_{s-})  \df s \nonumber.
\eeam
Using the compensation formula, Fact~\ref{fct:59drif} and the stationarity of $\vspr$ we get 
\begin{equation}\label{eq:EIVdsYd}
\begin{split}
\E \int_0^{\Delta} \calv_{s-}  \df \vect([L,L^{\ast}]_s) & = (\sigma_W + \sigma_L) \izede \E \calv_{s-} \vect(I_d) \df s = \Delta (\sigma_W + \sigma_L) \E \vect(V_0).
\end{split}
\end{equation}
Additionally, it follows from Lemma~\ref{le:momG1} that $\E \|\vect(V_s) \vect(Y_\Delta)^*\| < \infty$ for all $s \geq 0$ and that $s \mapsto \E \|\vect(V_s) \vect(Y_\Delta)^*\|$ is locally bounded. Then,
\begin{equation*}
\begin{split}
\quad\quad & \E \bigg(\int_0^{\Delta} \vect(V_{s-}) \df s \, (\vect(Y_{\Delta}))^{\ast} \bigg) = \izede \E \vect(V_{s-}) (\vect(Y_{\Delta}))^{\ast}  \df s \\
& = \Delta \vect(C) (\E \vect(Y_0))^{\ast} + \izede \E \vect(Y_{s}) (\vect(Y_{\Delta}))^{\ast}  \df s.
\end{split}
\end{equation*}
Now it follows from the invertibility of $(A \otimes A)$ and from the second equation following (3.5) in \cite{Stelzer10} that
\begin{equation}\label{eq:iVsLLd}
\begin{split}
\quad\quad &  \int_0^{\Delta} \calv_{s-}  \df \vect([L,L^{\ast}]^\disc_s) \\
& = (A \otimes A)^{-1} \bigg( \vect(Y_\Delta) - \vect(Y_0) - \int_0^{\Delta} (B\otimes I + I \otimes B) \vect(Y_{s-}) \df s \bigg).
\end{split}
\end{equation}
The representation in \eqref{eq:iVsLLd} gives
\beam
 & & \E \bigg[ \bigg( \int_0^{\Delta} \calv_{s-}  \df \vect([L,L^{\ast}]^\disc_s) \bigg) (\vect(Y_{\Delta}))^{\ast}  \bigg] \nonumber\\
& = & \E \bigg[ (A \otimes A)^{-1}  \Big( \vect(Y_\Delta) - \vect(Y_0) - \int_0^{\Delta} (B\otimes I + I \otimes B) \vect(Y_{s-}) \df s  \Big) (\vect(Y_{\Delta}))^{\ast} \bigg] \nonumber\\
& = & (A \otimes A)^{-1}  \bigg[ \E \vect(Y_\Delta)  (\vect(Y_\Delta))^{\ast} - 
\E \vect(Y_0)  (\vect(Y_\Delta))^{\ast} \nonumber\\
& & \quad -(B\otimes I + I \otimes B) \izede \E \vect(Y_{s-}) (\vect(Y_{\Delta}))^{\ast} \df s  \bigg] \label{eq:EIVsYde}.
\eeam
Using the definition of $C_\Delta$ in \eqref{eq:ipG1G1}, together with \eqref{eq:IVssplit}, \eqref{eq:EIVdsYd} and \eqref{eq:EIVsYde} gives
\beao
 & &  \cov(\vect (C_{\Delta}), \vect(Y_{\Delta})) = (A \otimes A)^{-1}  \bigg[\E \vect(Y_\Delta)  (\vect(Y_{\Delta}))^{\ast} - 
\E \vect(Y_0)  (\vect(Y_{\Delta}))^{\ast} \nonumber\\
& &  -(B\otimes I + I \otimes B) \bigg( \izede \E \vect(Y_{s-}) (\vect(Y_{\Delta}))^{\ast} \df s \bigg)  \bigg] \nonumber\\
& & + \Delta\sigma_W \vect(C) (\E \vect(Y_0))^{\ast} + \sigma_W \izede \E \vect(Y_{s-}) (\vect(Y_{\Delta}))^{\ast} \df s \nonumber\\
& &- \Delta(\sigma_W + \sigma_L) \E \vect(V_0) \E (\vect(Y_{\Delta}))^{\ast} \nonumber\\
& & = [\sigma_W I_{d^2} - (A \otimes A)^{-1}(B\otimes I + I \otimes B) ]  \izede \E \vect(Y_{s}) (\vect(Y_{\Delta}))^{\ast}  \df s \nonumber\\
& & + (A \otimes A)^{-1} \big[ \var(\vect(Y_0)) - \cov( \vect(Y_{0}), \vect(Y_{\Delta})) \big] - \Delta \sigma_L \vect(C) \E  (\vect(Y_{0}))^{\ast} \nonumber\\
& & - \Delta (\sigma_W + \sigma_L) \E \vect(Y_0) \E (\vect(Y_{0}))^{\ast}, \label{eq:cCYd}
\eeao
where the last equality follows from $V_0 = C + Y_0$ and the stationarity of $\yspr$. Using \eqref{eq:EV} it follows first that
\begin{equation}\label{eq:EYsYD}
\begin{split}
 \izede \E \vect(Y_{s}) (\vect(Y_{\Delta}))^{\ast}  \df s & = \izede e^{\calb(\Delta-s)} \var(\vect(Y_0)) \df s + \Delta \E \vect (Y_0) \E ( \vect (Y_0))^{\ast} \\
& = \calb^{-1} (e^{\calb \Delta} - I_{d^2}) \var(\vect(Y_0)) + \Delta \E \vect (Y_0) \E ( \vect (Y_0))^{\ast}, 
 \end{split}
\end{equation}
and second that
\begin{equation}\label{eq:subst}
 \var(\vect(Y_0)) - \cov( \vect(Y_{0}), \vect(Y_{\Delta})) = - (e^{\calb \Delta} - I_{d^2}) \var(\vect(Y_0)).
\end{equation}
Substituting $B\otimes I + I \otimes B = \calb - \sigma_L(A \otimes A)$, using \eqref{eq:EYsYD}, \eqref{eq:subst}, \eqref{eq:def:Calb} and the formula for $\E \vect(Y_0)$ in \eqref{eq:EV} gives
\begin{equation}\label{eq:cCdYdf}
\begin{split}
\quad\quad & \cov(\vect(C_{\Delta}), \vect(Y_{\Delta})) \\
& = \big[\sigma_W I_{d^2} - (A \otimes A)^{-1}(\calb - \sigma_L(A \otimes A)) \big]  \big[ \calb^{-1} (e^{\calb \Delta} - I_{d^2})\var(\vect(Y_0)) \\
& + \Delta \E \vect (Y_0) \E ( \vect (Y_0))^{\ast} \big] \\
& - (A \otimes A)^{-1} (e^{\calb \Delta} - I_{d^2}) \var(\vect(Y_0)) - \Delta \sigma_L \vect(C) \E  (\vect(Y_{0}))^{\ast} \\
& - \Delta (\sigma_W + \sigma_L) \E \vect(Y_0) \E (\vect(Y_{0}))^{\ast}\\
 & =  \big[ (\sigma_W + \sigma_L)\calb^{-1} - 2(A \otimes A)^{-1} \big] (e^{\calb \Delta} - I_{d^2}) \var(\vect(Y_0)) \\
 & \quad\quad - \big[ (A \otimes A)^{-1} \calb \E \vect (Y_0) + \sigma_L \vect(C)   \big] \Delta \E (\vect (Y_0))^{\ast} \\
& =  \big[ (\sigma_W + \sigma_L)\calb^{-1} - 2(A \otimes A)^{-1} \big] (e^{\calb \Delta} - I_{d^2}) \var(\vect(Y_0)) \\
 & \quad\quad - \big[ (A \otimes A)^{-1} \calb (-\sigma_L \calb^{-1} (A \otimes A) \vect(C)) + \sigma_L \vect(C)   \big] \Delta \E (\vect (Y_0))^{\ast} \\
& = \big[ (\sigma_W + \sigma_L)\calb^{-1} - 2(A \otimes A)^{-1} \big] (e^{\calb \Delta} - I_{d^2}) \var(\vect(Y_0)).
\end{split}
\end{equation}

Finally, the result of the Lemma follows from \eqref{eq:ipG1G1}, \eqref{eq:cvYAast}, \eqref{eq:cCdYdf} and the fact that
$$
\cov(\vect(Y_{\Delta}), \vect(G_{\Delta} G_{\Delta}^{\ast})) = (\cov( \vect(G_{\Delta} G_{\Delta}^{\ast}), \vect(Y_{\Delta})) )^{\ast} = 
( \cov(\vect(C_{\Delta}), \vect(Y_{\Delta})))^{\ast}.
$$
\end{proof}

\subsection{Proof of Lemma~\ref{le:acovgg}}

(i) The proof of Lemma~\ref{le:acovgg} (i) follows directly from Lemma~\ref{le:covYGG} combined with (5.7) in \cite{Stelzer10}.\\
(ii) Denoting by $\|\cdot \|_F$ the Frobenius norm we have by Lemma~\ref{le:momG1}(b) with $p = 2$
\beao
\E \|\vect(\bm{G}_1 \bm{G}_1^{\ast}) \vect(\bm{G}_1 \bm{G}_1^{\ast})^{\ast}\|_2 & = & \E \|\vect(\bm{G}_1 \bm{G}_1^{\ast})\|_2^2 = \E \|\bm{G}_1\bm{G}_1^{\ast}\|^2_F \\
& = & tr(\bm{G}_1\bm{G}_1^{\ast}\bm{G}_1\bm{G}_1^{\ast}) = \E \|\bm{G}_1\|_2^4 < \infty.
\eeao
Let $a_s := \vect(G_s G_s^{\ast}), s \in [0,\Delta]$ and use the integration by parts formula to write
\begin{equation}\label{eq:adads}
\begin{split}
a_{\Delta} a_{\Delta}^{\ast} & = \int_0^{\Delta} a_{s-} \df (a_s^{\ast}) +  \int_0^{\Delta} \df a_s  (a_{s-}^{\ast}) + [a,a^{\ast}]_{\Delta} \\
& = \bigg(  \int_0^{\Delta} \df a_s  (a_{s-}^{\ast}) \bigg)^{\ast} +  \int_0^{\Delta} \df a_s  (a_{s-}^{\ast}) +  [a,a^{\ast}]_{\Delta}
\end{split}
\end{equation}
hence we only need to prove that the random variables
\begin{equation*}
  \int_0^{\Delta} \df a_s  (a_{s-}^{\ast}) \quad \text{and} \quad   [a,a^{\ast}]_{\Delta},
\end{equation*}
have finite expectations and compute them in closed form. From \eqref{eq:ipG1G1}, Lemma~2.1(vi) in \cite{Behme2012} and the symmetry of $\vtpr$ it follows that 
\beam
\df a_t & = & \df (\vect(G_t G_t^{\ast})) \nonumber\\
& = & \df \bigg(\vect\bigg(  \int_0^{t} V_{s-}^{1/2} \df L_s G_{s-}^{\ast} + \int_0^{t} G_{s-} \df L_s ^{\ast}  V_{s-}^{1/2} + \int_0^{t}  V_{s-}^{1/2}  \df [L,L^\ast]_s  V_{s-}^{1/2}  \bigg) \bigg)\nonumber\\
& = & \df \bigg(   \int_0^{t} (G_{s-}\otimes \vshalf) \df L_s  + \int_0^{t} (\vshalf \otimes G_{s-}) \df L_s + \int_0^{t}  (\vshalf \otimes \vshalf)  \df \vect([L,L^{\ast}]_s)    \bigg) \nonumber\\
& = & (G_{t-}\otimes \vthalf + \vthalf \otimes G_{t-}) \df L_t + (\vthalf \otimes \vthalf)  \df \vect([L,L^{\ast}]_t), \quad t \geq 0 \label{eq:dat}. 
\eeam
By the sub-multiplicative property of $\|\cdot\|_2$, the generalized H\"older inequality with $(1/4 + 1/4 + 1/2 = 1)$ we have
\beam
& & \int_0^{\Delta} \E \|G_{s-}\otimes \vshalf \|_2^2 \| a_{s-} \|_2^2 \df s = \int_0^{\Delta} \E \|G_{s-}\|_2^2  \|\vshalf \|_2^2 \| \vect(G_{s-} G_{s-}^{\ast}) \|_2^2 \df s\label{eq:lbigs} \\
&= & \int_0^{\Delta} \E \|G_{s-}\|_2^2  \|\vshalf \|_2^2 \| G_{s-} G_{s-}^{\ast} \|_2^2 \df s \leq \int_0^{\Delta} (\E \|G_{s-}\|_2^8 )^{1/4} ( \E \|V_{s-}\|_2^4)^{1/4}  (\E \|G_{s-}\|_2^8 )^{1/2} \df s\nonumber,
\eeam
which is finite by Lemma~\ref{le:momG1} with $p = 4$. Additionally, similar calculations and Lemma~\ref{le:momG1} with $p = 2$ shows that
$\E (\|\vshalf \otimes \vshalf\|_2) \|a_{s-}\|_2 \leq  ( \E \|V_{s-}\|_2^2)^{1/2}  (\E \|G_{s-}\|_2^4 )^{1/2} < \infty $ for all $s > 0$ and the map $s \mapsto \E (\|\vshalf \otimes \vshalf\|_2 \|a_{s-}\|_2)$ is locally bounded. Thus it follows from \eqref{eq:dat}, the It\^o isometry, the fact that $[L,L^\ast]_t = [L,L^\ast]_t^\disc + \sigma_w I_d t$ and fact~\ref{fct:compGen} that
\begin{equation}\label{eq:Eidasas}
\begin{split}
\quad\,\, & \E  \int_0^{\Delta} \df a_s  (a_{s-}^{\ast}) \\
& = \E \bigg( \int_0^{\Delta} (G_{s-}\otimes \vshalf + \vshalf \otimes G_{s-}) \df L_s a_{s-}^{\ast} + \int_0^{\Delta} (\vshalf \otimes \vshalf)  \df ( \vect([L,L^{\ast}]_s) a_{s-}^{\ast}  \bigg) \\
& = (\sigma_L + \sigma_W) \bigg(  \int_0^{\Delta} \E \big( (\vshalf \otimes \vshalf) \vect(I_d) a_{s-}^{\ast} \big) \df s \bigg) \\
& = (\sigma_L + \sigma_W) \int_0^{\Delta} \E ( \vect(V_{s-})a_{s-}^{\ast}  ) \df s.
\end{split}
\end{equation}
It follows from (5.6) in \cite{Stelzer10} that 
\begin{equation}\label{eq:iEas}
\izede \E a_{s-}^{\ast}  \df s = \izede \big( \vect((\sigma_L + \sigma_W)s \E V_0) \big)^{\ast} \df s = \frac{1}{2} (\sigma_L + \sigma_W) \Delta^2 \E \vect( V_0)^{\ast}.
\end{equation}
Since we assumed here that all hypothesis for using Lemma~\ref{le:covYGG} are valid, we can use \eqref{eq:lecovYG} with $\Delta = s$ to get
\begin{equation}\label{eq:icovYa}
\begin{split}
\quad\quad &  \izede \cov(\vect(Y_{s-}), a_{s-})  \df s  \\
& = \var (\vect (Y_0)) \bigg( \izede    (e^{\calb^{\ast}s} - I_{d^2})  \df s \bigg) \big[ (\sigma_W + \sigma_L)(\calb^{\ast})^{-1} - 2((A \otimes A)^{\ast})^{-1} \big] \\
& = \var (\vect (Y_0)) \tilde{\calb},
\end{split}
\end{equation}
where $\tilde{\calb}$ is defined in \eqref{eq:Btilde}.
Using \eqref{eq:Eidasas}, \ref{as:YtsecSt} \eqref{eq:iEas}, \eqref{eq:icovYa} gives 
\begin{equation}\label{eq:IEasds}
\begin{split}
\int_0^{\Delta} \E  \vect(V_{s-})a_{s-}^{\ast}   \df s & =  \izede  \cov(\vect(V_{s-}), a_{s-})  \df s  +   (\E \vect(V_s))\izede \E (a_{s-}^{\ast})  \df s  \\
& =    \izede  \cov(\vect(Y_{s-}), a_{s-})  \df s + 
(\E \vect(V_0))\izede \E (a_{s-}^{\ast})  \df s  \\
& = \frac{1}{2}(\sigma_L + \sigma_W)\Delta^2\E \vect(V_0)  E  \vect( V_0)^{\ast}  +  \var (\vect (Y_0)) \tilde{\calb} \\
& =  (\sigma_L + \sigma_W)^{-1} D,
\end{split}
\end{equation}
where $D$ is defined in \eqref{eq:E}.
Let $f_s:= (G_{s-}\otimes \vshalf + \vshalf \otimes G_{s-}), s \geq 0$ and recall $\calv_{s-} = \vshalf \otimes \vshalf$. Using \eqref{eq:ipG1G1}, Lemma~2.1(vi) in \cite{Behme2012} and the symmetry of $\vshalf$ gives
\begin{equation}\label{eq:i1234}
\begin{split}
\quad\quad\,\, & [a,a^{\ast}]_{\Delta} \\
& = \bigg[ \vect\bigg( \int_0^{\cdot} V_{s-}^{1/2} \df L_s G_{s-}^{\ast} + \int_0^{\cdot} G_{s-} \df L_s ^{\ast}  V_{s-}^{1/2} + \int_0^{\cdot}  V_{s-}^{1/2}  \df \qvl_s  V_{s-}^{1/2} \bigg), \\
& \bigg(\vect\bigg( \int_0^{\cdot} V_{s-}^{1/2} \df L_s G_{s-}^{\ast} + \int_0^{\cdot} G_{s-} \df L_s ^{\ast}  V_{s-}^{1/2} + \int_0^{\cdot}  V_{s-}^{1/2}  \df \qvl_s  V_{s-}^{1/2} \bigg)\bigg)^{\ast} \bigg]_{\Delta} \\
& = \bigg[  \int_0^{\cdot} f_{s-} \df L_s + \int_0^{\cdot}  \calv_{s-}  \df \vect(\qvl_s) , \int_0^{\cdot} \df L_s^{\ast} f_{s-}^{\ast} +  \int_0^{\cdot} \df (\vect(\qvl_s)^{\ast})  \calv_{s-} \bigg]_{\Delta} \\
& = \izede f_{s-} \df\qvl_s  f_{s-}^{\ast} + \izede  f_{s-}  \df[L, \vect(\qvl)^{\ast}]_s\calv_{s-} \\
& +   \izede \calv_{s-} \df[\vect(\qvl), L^{\ast} ]_s  f_{s-}^{\ast} +  \izede  \calv_{s-}  \df[ \vect(\qvl), \vect(\qvl)^{\ast} ]  \calv_{s-} \\
& := I_1 + I_2 + I_3 + I_4.
\end{split}
\end{equation}
By Lemma~\ref{le:momG1} with $p =2$ and similar calculations as in \eqref{eq:lbigs} it follows that $\E \|\calv_{s-}\| \| f_{s-} \| < \infty$ for all $s > 0$ and the map $s \mapsto \E \|\calv_{s-}\| \|f_{s-} \|$ is locally bounded. Thus, it follows from \ref{as:lp_qv} that we have $\E I_2 = \E I_3 = 0$. Now, Lemma~\ref{le:momG1} gives $\E \|\calv_{s-}\|^2 < \infty$ for all $s > 0$ and local boundedness of the map $s \mapsto \E \|\calv_{s-}\|^2$. Using the second-order stationarity of $\vspr$ in \ref{as:YtsecSt}, the compensation formula and the formulas at p. 108 in \cite{Stelzer10}
\begin{equation}\label{eq:Ei4}
\begin{split}
\E I_4 & = \E \bigg( \izede  \calv_{s-} \df[ \vect(\qvl), \vect(\qvl)^{\ast} ]  \calv_{s-} \bigg) \\
& = \E \bigg( \izede  \calv_{s-} \df[ \vect(\qvl^\disc), (\vect(\qvl^\disc))^{\ast} ]^\disc  \calv_{s-} \bigg) \\
& = \izede \E \big( \calv_{s-}  \rho_L [ I_{d^2} + K_{d} + \vect(I_d) \vect(I_d)^{\ast} ]  \calv_{s-}  \big) \df s \\
& =  \rho_L \izede (\bsQ +K_{d} \bsQ  +  I_{d^2}) \E( \vect(V_s) \vect(V_s)^{\ast} ) \df s \\
& = \Delta \rho_L  (\bsQ +K_{d} \bsQ  +  I_{d^2}) \E \vect(V_0) \vect(V_0)^{\ast}.
\end{split}
\end{equation}
To compute $\E I_1$ we will need the following matrix identity, which is based on Fact~7.4.30 (xiv) in \cite{Bernstein05}. Let $A \in \mdone$ and $B,B^2 \in \mdd$ be symmetric matrices. Then, 

\begin{equation}\label{eq:matrFac}
\begin{split}
\quad\quad  & (A \otimes B + B \otimes A)(A \otimes B + B \otimes A)^\ast = (A \otimes B + K_d(A \otimes B) ) (A \otimes B + K_d(A \otimes B) )^\ast \\
& = (I + K_d) (A \otimes B) (A^\ast \otimes B) (I + K_d) = (I + K_d) \bsQ \vect(AA^{\ast}) \vect(B^2) (I + K_d).
\end{split}
\end{equation}

Write $b_s:= \E \vect(G_s G_s^{\ast}) \vect(V_s)^{\ast}$, which is finite by Lemma~\ref{le:momG1} with $p = 2$. Using the compensation formula, \eqref{eq:matrFac} and the definition of $f_s$ gives
\begin{equation}\label{eq:Efsdllf}
\begin{split}
\quad\,\, & \E  \bigg( \izede f_{s-} \df[L,L^{\ast}]_s  f_{s-}^{\ast} \bigg) \\
& = (\sigma_L + \sigma_W) \izede \E (f_s f_s^{\ast}) \df s \\
& = (\sigma_L + \sigma_W) \izede \E (G_{s-} \otimes \vshalf + \vshalf \otimes G_{s-})(G^{\ast}_{s-} \otimes \vshalf + \vshalf \otimes G_{s-}^{\ast}) \df s \\ 
& =  (\sigma_L + \sigma_W)\izede (I + K_d) \bsQ b_s (I + K_d) \df s \\
& =  (\sigma_L + \sigma_W) (I + K_d) \bsQ \bigg( \izede   b_s \df s \bigg)  (I + K_d).
\end{split}
\end{equation}
Finally, it follows from \eqref{eq:IEasds} that
\begin{equation}\label{eq:Ecalc}
\izede b_s^{\ast} \df s  =   \int_0^{\Delta}  \E \vect(V_s) a_{s-}^{\ast} \df s =  (\sigma_L + \sigma_W)^{-1} D.
\end{equation}

The result now is a direct consequence of \eqref{eq:adads}, \eqref{eq:IEasds}, \eqref{eq:i1234}, \eqref{eq:Ei4}, \eqref{eq:Efsdllf} and \eqref{eq:Ecalc}.

\begin{remark}\label{re:new}
An inspection of the proofs of Lemmas~\ref{le:acovgg} and  \ref{le:covYGG} shows that the moment assumptions \ref{as:El18fin} and \ref{as:Y0E4} are only needed to compute expectations of stochastic integrals with the integrator $L$. If $L$ has paths of finite variation, these expectations can be computed by using the compensation formulas given in Facts~\ref{fct:59drif} and \ref{fct:compGen} without \ref{as:El18fin} and \ref{as:Y0E4}.
\end{remark}

\begin{small}
\bibliographystyle{plainnat}
\bibliography{references}       
\end{small}
       
\end{document}